\documentclass[11pt]{article}
\usepackage{amsmath,amssymb,amsthm}
\usepackage{graphicx}
\usepackage{array}   
\newcolumntype{C}{>{$}c<{$}} 

\usepackage{xcolor}

\theoremstyle{plain}
\newtheorem{theorem}{Theorem}
\newtheorem*{theorem*}{Theorem}
\newtheorem{lemma}{Lemma}
\newtheorem{corollary}{Corollary}
\newtheorem*{corollary*}{Corollary}
\newtheorem{definition}{Definition}

\newtheorem*{conjecture*}{Conjecture}

\theoremstyle{definition}
\newtheorem{example}{Example}
\newtheorem*{remark*}{Remark}

\DeclareMathOperator*{\bigtimes}{\scalebox{1.5}{$\times$}}
\DeclareMathOperator*{\join}{\scalebox{1.8}{$\Join$}}

\def\beq{\begin{equation}}
\def\eeq{\end{equation}}

\def\R{{\mathbb R}}
\def\N{{\mathbb N}}

\def\Q{{\mathbb Q}}

\def\-{\backslash}
\def\A{\mathcal{A}}
\def\D{\mathcal{D}}
\def\P{\mathcal{P}}
\def\L{\Lambda}
\def\K{\textup{K}}
\def\M{\textup{M}}
\def\T{\textup{T}}
\def\S{\Sigma}
\def\s{\sigma}
\def\a{\alpha}
\def\b{\beta}
\def\g{\gamma}
\def\hR{\widehat{R}}
\def\ter{\mathrm{\mathbf{ter\,}}}

\title{Logical reduction of relations:
from relational databases to Peirce's reduction thesis}
\author{Sergiy Koshkin\\
\\
Department of Mathematics and Statistics\\
University of Houston-Downtown\\
One Main Street\\
Houston, TX 77002\\
e-mail: koshkins@uhd.edu}
\date{}
\begin{document}

\maketitle
\begin{abstract} 

We study logical reduction (factorization) of relations into relations of lower arity by Boolean or relative products that come from applying conjunctions and existential quantifiers to predicates, i.e. by primitive positive formulas of predicate calculus. Our algebraic framework unifies natural joins and data dependencies of database theory and relational algebra of clone theory with the bond algebra of C.S. Peirce. We also offer new constructions of reductions, systematically study irreducible relations and reductions to them, and introduce a new characteristic of relations, ternarity, that measures their `complexity of relating' and allows to refine reduction results. In particular, we refine Peirce's controversial reduction thesis, and show that reducibility behavior is dramatically different on finite and infinite domains.

\bigskip

\end{abstract}

\section*{Introduction}

We study decomposition of relations into simpler relations by logical operations expressible in terms of conjunctions and existential quantifiers on predicates, i.e. by primitive positive formulas of predicate calculus \cite{Lag17}. Such analysis goes back to the work of C.S. Peirce at the dawn of algebraic logic, see \cite{Brun91,Bur97,Kosh22} for modern accounts. However, while some aspects of it have been developed by Schr\"oder, L\"owenheim, Tarski, and others \cite{Buss}, they gradually shifted the focus to predicates in formal theories and questions of axiomatization. On the other hand, decomposition of relations in algebras of relational clones or co-clones is closely related to Peirce's, but was developed independently of his work, from Post's study of the dual decomposition of functions on finite sets, see  \cite{Lau,PosKal}. Peirce's results were largely forgotten until modern formalizations of his bond algebra in \cite{Bur91,Herz}.

The original interest was linguistic and philosophical. A different stimulus came from the relational database model introduced by Codd \cite{Codd70}, and developed by Fagin \cite{Fag77}, Rissanen \cite{Riss77}, and others. It featured weaker than Peirce's, but closely related, notion of decomposition. Unfortunately, the two streams of literature remained largely disjoint. This paper is, in part, a unified survey of classical but little known (to logicians and mathematicians) results and their interconnections, and, in part, their extension by the author.

From the mathematical perspective, join decomposition of relations is somewhat analogous to factorization of polynomials. There is even a ready analog of polynomial's degree that measures its `multiplicative complexity' -- relation's arity (adicity, rank), the number of its attributes (places, positions). There is an even stronger analogy to the recent theory of factorization in highly non-cancellative monoids with degree-like height function \cite{Tring22}. Accordingly, we will call decomposition terms {\it factors}, and call decompositions {\it reductions} when all the factors have strictly lower arity than the relation itself. 

A technical device of database theory that we will extensively use in this paper is the calculus of {\it attributed relations} (the term is due to \cite{YanPap}, see also {\it named perspective} in \cite[3.2]{AHV95}). Those are collections of maps from a set of attributes, called the relation scheme, to the domain, rather than ordered lists of domain elements, as in ordinary set-theoretic relations. This allows us to identify positions in different relations through relation schemes, the same flexibility one gets by placing the same variable into different predicates in logical formulas, and removes much of combinatorial clutter that plagues definitions and calculus of operations on ordinary relations (compare to \cite{Bur91,CorPos04}). Moreover, operations on attributed relations have better algebraic properties.

For example, the Cartesian product is commutative and associative on attributed relations, albeit only partially defined: relations that share attributes cannot be multiplied. What we get by extending it to all attributed relations is classically known as Boolean product, and it is still commutative and associative. In logical terms, it amounts to taking conjunctions of predicates with some variables identified instead of only free conjunctions. In predicate calculus, such expressions are called quantifier-free primitive positive formulas \cite{Lag17}. Its iteration, called the natural join, is the primary operation of the database theory. The reason is that join decompositions are closely associated with various data dependencies, and can be used to store, query and update the data more efficiently \cite{ABU79,BeeFag81,Fag77,Lee83,Maier,Riss77}.

We start by reviewing some known join reduction methods that rely on exploiting various dependencies among relation's tuples. The simplest one is functional dependency, when there is an attribute, called {\it key}, whose value determines the rest of its tuple (ID columns play this role in database tables). However, despite the abundance of reduction methods, it turns out that join irreducible relations also abound. Not only are there ones of arbitrarily high arity, but also `almost all' relations are join irreducible on large domains (Theorem \ref{JoinRedLim}). Although this fact must have been known to experts, the author did not encounter it stated in print.

Thus, join reduction is analogous to factorization of polynomials over the field of rationals, with irreducibles of arbitrarily high degrees. This prompts one to look for extensions with a more manageable set of irreducibles. However, in line with the literature on the subject, rather than enlarging the class of relations we opt for stronger algebra on the same class. The stronger operation is projective join (projoin for short) that combines join with projections. In logical terms, we allow to existentially quantify (`project out') some variables in predicate conjunctions, i.e. consider all primitive positive predicate formulas. 

While some project-join algebras have been studied in the more theoretical database literature \cite{DunMik,YanPap}, they appeared much more prominently in the theory of constraint satisfaction problems (CSP) that also dates back to 1970s. The constraints are expressed by relations on finite domains, and the problems are to decide whether relations from a given set can simultaneously hold on some elements of the domain (be satisfied). Computational complexity of CSP was actively studied in computer science, and in 1978 Schaefer discovered a deep connection between it and closure properties of sets of relations under projective joins on $2$-element domains \cite{Sch78}. The work initiated by Feder-Vardi and Jeavons in 1990s extended this connection to all finite domains, giving rise to what is now called the algebraic approach to complexity of CSP \cite{Born08}. In particular, it turned out that  if constraints are projoin complete, i.e. generate all relations, then the CSP is NP-complete (assuming P$\neq$NP). In general, projoin closures of relations are called relational clones or co-clones, and Schaefer characterized NP completeness in terms of them on $2$-element domains. They are dual to the functional clones of Post and studied since 1960s \cite{Lau,PosKal}. Join closures, called weak partial co-clones, also found applications to studying complexity of CSP, namely to refined complexity classification of NP-complete problems \cite{JLNZ13}.

The projoin reduction problem on finite domains can be reformulated as asking whether the set of all unary and binary relations is projoin complete, i.e. whether they generate the co-clone of all relations. In hindsight, the affirmative answer to the same question on infinite domains goes back to Peirce. The proof device, introduced by Peirce under the quaint name of `hypostatic abstraction', amounts, in database terms, to attaching key attribute(s) to the relation, and then projecting them out after join reducing the augmented relation (Theorem \ref{ProKeyRed}). Reducibility of all relations to unary and binary ones is more analogous to factorization of polynomials over the field of real numbers, where they reduce to linear and quadratic factors. 

Counterintuitively, the situation is much more complex on finite domains, and we introduce some further devices that allow to projoin reduce some relations when hypostatic abstraction does not (Section \ref{finprojoin}). They exploit the connection between projections and unions (existential quantification and disjunction). In particular, we generalize to projoins Fagin's characterization of certain joins in terms of multivalued dependencies (Theorem \ref{ProMultKeyRed}), and show how to convert complements (negations) of certain joins into projoins. However, we do not resort to a central device of clone theory, the Pol-Inv Galois connection \cite{Born08,Lau}, as our approach is to develop the more elementary Peircean methods that construct reductions explicitly.

While projection and join can be conveniently folded into a single projoin operation, this operation is not an iteration of any binary operation, like Cartesian product and join were. However, one can iterate a particular binary projoin classically known as relative product \cite{Brun91} (composition of binary relations is its restriction to them). Relative product is associative only on a restricted class of factors, and generates only an (ostensibly) narrow subclass of projoins, which we call bonds following \cite{Herz}. 

The bond algebra is, more or less, Peirce's project-join algebra. In a surprise, he established that bond reducibility is almost equivalent to general projoin reducibility. All relations of arity $4$ and higher are projoin reducible if and only if they are bond reducible, and it is only on ternaries (ternary relations) that the two notions diverge (Theorem \ref{ProJointoBond}). There are some projoin reducible but bond irreducible ternaries, notably the teridentity relation $I_3$ that contains all and only identical triples of domain elements and plays a key role in converting projoins into bonds. Bond reducibility of all relations on infinite domains to unaries, binaries and ternaries is known as Peirce's reduction thesis (see Theorem \ref{NotInProRedEq} for a precise formulation), and it is equivalent to the better known projoin reducibility to binaries alone, see e.g. \cite{Low}. The seeming discrepancy fueled a long historical controversy \cite{Bur97,Kosh22}.

Motivated by the thesis, we introduce the notion of ternarity of a relation as the minimal number of ternaries in its bond reductions to unaries, binaries and ternaries, and study its properties. The reason for singling out ternaries is that they do the main work of `relating' different attributes in a relation as represented by reductions. We establish a close connection between complete bond reductions and subcubic graphs, and prove that ternarity of non-degenerate $n$-ary relations is always $n-2$ on infinite domains (Theorem \ref{TerNonDeg}), one unit of ternarity per unit of arity over $n=2$. This refines the original reduction thesis. The proof uses graph-theoretic methods, pioneered by Peirce and prominent in the recent work on reduction \cite{CorDau,CorPos04}. A counterexample then shows that on finite domains this equality can fail already for $n=4$, and we prove that it fails for `almost all' relations on large finite domains.

The paper is organized as follows. Section \ref{Prelim} introduces our terminology and notation. In Sections \ref{Cart}-\ref{join} we review some standard results on Cartesian products and natural joins, stated in terms of attributed relations, and add results on join irreducibility. Sections \ref{projoin}-\ref{finprojoin} introduce projective joins (projoins) and associated reduction methods, old and new, including the hypostatic abstraction that settles the reduction problem on infinite domains. The irreducibility results are weaker for projoins, and only concern some restricted classes of them. We also explain why reduction behavior is so different in finite and infinite cases. In Section \ref{bond} we introduce Peirce's bonds, and explain his algorithm for converting projoins into them that features teridentity. Peirce's reduction thesis is then derived as a corollary of the results for projoins. Section \ref{diag} defines projoin graphs and bonding diagrams that allow us to apply graph-theoretic methods to bond reductions in Section \ref{tern}, which studies complete reductions and introduces ternarity. In Section \ref{ternPRT} we use ternarity to refine Peirce's reduction thesis and give some counterexamples. In the last section, we summarize our conclusions and state some open problems.

\section{Preliminaries}\label{Prelim}

We use the standard set-theoretic notation and terminology for sets and relations \cite{Lip}. Relations are defined on a set $\D$ called the {\it domain} and are subsets of its Cartesian powers $\D\times\dots\times\D$. When $R\subseteq\D^n$ the number $n$ is called the relation's {\it arity} and the relation is called $n$-ary relation or simply {\it $n$-ary} used as a noun. For $n=1,2,3,4$ we use the shorthands unary, binary, ternary, quaternary, respectively. 

Elements of $\D^n$ are called $n$-tuples or just {\it tuples}, when $n$ is understood or immaterial. If $a\in\D^n$ it's $i$-th member is denoted $a_i$. We adopt the usual convention of canonically identifying tuples of tuples with longer tuples, and hence of identifying $\D^n\times\D^{m}$ with $\D^{n+m}$, and so on. It is often convenient to interpret tuples as maps from the set of relation's positions $\N_n:=\{1,2,\dots,n\}$ to $\D$. 

Some standard $n$-aries that can be defined on any domain will be called and denoted as follows: the {\it empty relations} $\emptyset_n$ with no tuples; the {\it universal relations} $U_n:=\D^n$ that contain all possible $n$-tuples; the {\it identity relations} $I_n$ that contain all and only $n$-tuples with identical members; and the {\it diversity relations} $D_n$ that contain all and only $n$-tuples with pairwise distinct members. Each of those can be relativized to proper subsets $\A\subset\D$ that we will indicate by the upper index, e.g. $I_n^\A$ contains all and only $n$-tuples with identical members from $\A$.

The set of all maps from $S$ to $\D$ is denoted $\D^S$\!, the set of all subsets of $S$ is denoted $\P(S)$, and its cardinality  is denoted $|S|$. As is well known, $|\D^n|=|\D|^n$, $|\D^S|=|\D|^{|S|}$, and $|\P(S)|=2^{|S|}$.

{\it Attributed relation} $R$ is a subset $R\subseteq\D^\S$ \cite{CorDau,DunMik,YanPap},  where $\S$ is a set called the {\it relation scheme}, and its elements are called {\it attributes}. The {\it arity} of $R$ is defined to be $|\S|$, and its elements, which are now functions from $\S$ to $\D$, are also called tuples. We only consider finitary relations, so $\S$ is always finite, the ordinary relations correspond to $\S=\N_n$. When $\S$ is linearly ordered, e.g.  $\S\subset\N$, there is a canonical $1$-$1$ correspondence between $\S$ and $\N_{|\S|}$ obtained by listing elements of $\S$ in order, which induces a canonical $1$-$1$ correspondence between relations on the scheme $\S$ and ordinary relations.

When $\L\subseteq\S$ we will denote $a_\L\in\D^\L$ the attributed subtuple of $a$ consisting of its members in the $\L$ positions. When $\L,\M\subseteq\S$, and $\a\in\D^\L$, $\b\in\D^\M$ with $\a_{\L\cap\M}=\b_{\L\cap\M}$, we define their {\it concatenation} $\a\mid\b\in\D^{\L\cup\M}$ as the union of ordered pairs when they are taken as functions from $\S$ to $\D$. The intersection condition is needed for the union to also be a function, and the concatenation is always defined when $\L,\M$ are disjoint. When $\S\subset\N$ the members of the concatenated tuple are listed according to the order of their attributes in $\N$. For example, $((\a_1,\a_3)\mid(\b_2,\b_3,\b_5))=(\a_1,\b_2,\a_3,\b_5)$ assuming $\a_3=\b_3$. 

In the database model relations are visualized as rectangular tables with columns labeled by positions (attributes) and rows listing the tuple members. Database queries are then interpreted as operations on relations that extract relevant information and package it into simpler relations. Two operations inspired by this interpretation will be useful to us. The first is {\it projection} to a subset of positions $\L\subseteq\S$:
$$
\pi_\L R:=\{a_\L\,\Big|\,a\in R\},
$$
that simply deletes all non-$\L$ columns and removes duplicate tuples, if any, in the remaining $\L$ columns. 
When no confusion results, we write simply $\pi_{i_1,\dots,i_k}R$ instead of $\pi_{\{i_1,\dots,i_k\}}R$.
The second operation is {\it selection} over a subset:
\begin{equation}\label{ValSel}
\s_{x_\L=\a}R:=\{a_{\L^c}\,\Big|\,a\in R, a_\L=\a\},
\end{equation}
that leaves only rows with prescribed values in the $\L$ columns (given by $\a\in\D^\L$), and then deletes those columns. Here $\L^c:=\S\-\L$ is the complement of $\L$. Note that both $\pi_\L R$ and $\s_{x_\L=\a}R$ are attributed relations on the schemes $\L,\L^c$, respectively.

Unless otherwise stated, all predicates will be interpreted on a domain, and relations will be identified with the predicates they interpret. In particular, the same letter will be used for a relation and its predicate, i.e. $R(a_1,\dots,a_n)$ will mean the same as $(a_1,\dots,a_n)\in R$, and $\neg R$ will denote the complement of $R$ in $\D^\S$. This identification is convenient because attribution maps are naturally expressed in predicates by placing the same variables into multiple positions. 

The same notational conventions apply to tuples of variables as to tuples of values, e.g. $x_\L$ is the subtuple of variables with the indices from $\L$, $x_\L\mid x_M$ is the concatenation of variable tuples, etc. The standard logical operations, conjunction $\land$, disjunction $\lor$, etc., will be used with the usual meaning, and for $\L=\{i_1,\dots,i_k\}$ the multiple quantification $\exists x_{i_1}\dots\exists x_{i_1}R(x)$ will be abbreviated as $\exists x_\L R(x)$. With these conventions and in terms of predicates, the projection is expressed simply as $\pi_\L R\,(x_\L)=\exists x_{\L^c} R(x)$, and the selection as $\s_{x_\L=\a}R\,(x_{\L^c})=R(\a\!\mid\! x_{\L^c})$. For example, $\pi_{1,3}R\,(x,y)=\exists tR(x,t,y)$, and $\s_{x_{1,3}=(\a',\a'')}R\,(z)=R(\a',z,\a'')$.

\section{Cartesian Products}\label{Cart}

The simplest way a relation decomposes into lower arity ones is when it is a Cartesian product of them. We need a slight generalization so that the relation does not cease to be a Cartesian product simply because its positions are permuted. This is taken care of when using attributed relations. We give the definition directly for any finite number of factors, but one can see that it comes from iterating the binary product. Recall that a {\it partition} $\S=\L_1\cup\dots\cup\L_m$ of a set $\L$ is its representation as a disjoint union of subsets. 
\begin{definition}\label{CartPart}
Given disjoint finite subsets $\L_i\subseteq\S$ and attributed relations $R_i\subseteq\D^{\L_i}$, their Cartesian product $R\subseteq\D^{\,\cup_i\L_i}$ is the set of concatenations of tuples from $R_i$:
$$
R_1\times\dots\times R_m:=\{(a^1|\dots|\, a^m)\,\Big|\,a^i\in R_i\}.
$$
A relation $R\in\D^\S$ is a \textbf{Cartesian product over a partition} $\S=\L_1\cup\dots\cup\L_m$ when there exist $R^{\L_i}\in\D^{\L_i}$ such that $R=R^{\L_1}\times\dots\times R^{\L_m}$, and $R^{\L_i}$ are called its \textbf{Cartesian factors}. It is called \textbf{degenerate} when it is a Cartesian product over a partition with $m>1$ and all $\L_i\neq\emptyset$, and it is called \textbf{$\boldsymbol{l}$-aric} when all factors have the same arity $l$.
\end{definition}
\noindent We reiterate that, on this definition, Cartesian factors are attributed relations, and $R$ is not their Cartesian product in the ordinary sense, but rather some permutation of it. By definition, Cartesian factorization is a reduction because all $R^{\L_i}$ have smaller arity $|\L_i|<n$, and when it exists, the factors are none other than the projections $R^{\L_i}=\pi_{\L_i} R$.

The definition is more straightforward in terms of predicates, it means that the predicate of $R$ is a {\it free conjunction} of lower arity predicates:
\begin{equation}\label{FreeConj}
R(x_1,\dots,x_n)=R^{\L_1}(x_{\L_1})\land\dots\land R^{\L_m}(x_{\L_m}).
\end{equation} 
The following is a characteristic property of Cartesian products.
\begin{theorem}[\textbf{Independence criterion}]\label{CartInd} 
A relation $R\subseteq\D^\S$ is a Cartesian product over a partition $\S=\L_1\cup\dots\cup\L_m$ if and only if the values of its tuples on $\L_i$ can be chosen independently, i.e. for any collection of $\a^i\in\pi_{\L_i}R$ there exists a common $a\in R$ such that $a_{\L_i}=\a^i$.
\end{theorem}
\begin{proof} Let $R$ be a Cartesian product over $\L_i$. If $\a^i\in\pi_{\L_i}R=R^{\L_i}$ then $R^{\L_i}(\a^i)$ holds for all $i$, and hence $R^{\L_1}(\a^1)\land\dots\land R^{\L_m}(\a^m)$ holds. But by \eqref{FreeConj} this means that $R(a)$ holds for the concatenation $a:=(\a^1|\dots|\,\a^m)$. Thus, $a\in R$ and $a_{\L_i}=\a^i$.

Conversely, suppose that $R$ satisfies the independence condition over $\L_i$. If $a\in R$ then $a_{\L_i}\in\pi_{\L_i} R$ by definition of projection, so $R\subseteq\pi_{\L_1}R\times\dots\times\pi_{\L_m}R$. On the other hand, if $a\in\pi_{\L_1}R\times\dots\times\pi_{\L_m}R$ then $a_{\L_i}\in\pi_{\L_i} R$, and, by the independence condition, there is a common $b\in R$ with $b_{\L_i}=a_{\L_i}$. Since $\L_i$ form a partition we have $a=b\in R$. Thus, $\pi_{\L_1}R\times\dots\times\pi_{\L_m}R\subseteq R$ and they are equal.
\end{proof}
\noindent Note that if we assign equal probability to all tuples of $R$ degenerate over a partition $\L_i$ then $a_{\L_i}$ will be statistically independent random vectors on this sample space \cite{Mal}.

When testing for degeneracy it is sufficient to consider bipartitions only. Indeed, if $R$ can be split into several factors then we can always group them into just two non-trivial factors, say, the first one and the product of the rest.
\begin{example} The independence criterion can be used directly to establish degeneracy or non-degeneracy of some relations. Clearly, all unary relations are non-degenerate on any domain. On domains $\D$ with $|\D|\geq2$ the identity relations $I_n$ are non-degenerate for $n\geq2$. Indeed, constant tuples $(\a,\dots,\a),(\b,\dots,\b)\in I_n$ for any $\a,\b\in\D$. If $I_n$ were degenerate, we could, by independence, assign constant $\a$ values to some $\L_i$ and constant $\b$ values to others. But if $\a\neq\b$ their concatenation will not be in $I_n$, contradiction. 

Similarly, for $|\D|\geq n$ the diversity relations $D_n$ are non-degenerate ($D_n=\emptyset_n$ for $|\D|<n$). Indeed, any diverse (with no equal values) subtuple is in the projection of $D_n$ to any proper subset of positions. But subtuples over disjoint subsets of positions can share values, and their concatenation will not be in $D_n$.
\end{example}

The following simple proposition gives two sufficient conditions of non-degeneracy that are easier to check in examples.
\begin{theorem}\label{DegTests} Suppose one of the following conditions holds for a relation $R\subseteq\D^\S$.

\textup{(i)} $R$ is not universal, but for any non-empty proper subset $\emptyset\subset\L\subset\S$ \par we have $\pi_{\L}R=\D^{\L}$. 

\textup{(ii)} $\neg R$ is not empty, but for any $i\in\S$ we have $\pi_i(\neg R)\neq\D$.

Then $R$ is non-degenerate.

\end{theorem}
\begin{proof} \textup{(i)} According to \eqref{FreeConj}, $R$ is a conjunction of $\pi_{\L_i} R$ with proper subsets $\L_i$. Since each projection is universal, by assumption, so must be their conjunction, contradiction. 

\textup{(ii)} Let $\delta\in\D$ be arbitrary. If $R$ is a product negating \eqref{FreeConj} we obtain, by the de Morgan law:
\begin{equation*}
\neg R(x_1,\dots,x_n)=\neg R^{\L_1}(x_{\L_1})\lor\dots\lor\neg  R^{\L_m}(x_{\L_m}).
\end{equation*} 
Since $\neg R$ is not empty so is one of the disjuncts, say $\neg R^{\L_1}$. Pick some $\a\in\neg R^{\L_1}$ and some $i\not\in\L_1$, which exists because $\L_1$ is proper, and set $a_i:=\delta$, $a_{\L_1}:=\a$, then assign values for $a_j$ with $j\not\in\L_1\cup\{i\}$  arbitrarily. By construction, $\neg R^{\L_1}(a_{\L_1})$ holds and hence so does $\neg R(a)$ because $\neg R$ is a disjunction of $\neg R^{\L_i}$. But $\delta$ was arbitrary, and so $\pi_i(\neg R)=\D$, contrary to the assumption.
\end{proof}
\begin{example}\label{DegNotIn} By Theorem \ref{DegTests}\,(i), the non-identity relations $\neg I_n$ with any $n\geq2$ are non-degenerate (they trivially are for $n=1$) when $|\D|\geq2$. Indeed, $\pi_{\L}(\neg I_n)$ is universal for any non-empty proper $\L$ because any tuple constant over $\L$ can be complemented by assigning a different value on $i\not\in\L$, and non-constant tuples can be complemented arbitrarily, with the result in $\neg I_n$ in both cases.

The above reasoning no longer applies to $R:=\neg I_n^\A$ for $\A$ a non-empty proper subset of $\D$. However, $\pi_i(\neg R)=\pi_i(I_n^\A)=\A\neq\D$ for any $i$ by assumption, and $\neg I_n^\A$ is non-degenerate by Theorem \ref{DegTests}\,(ii). In particular, by picking $\A=\{\a\}$ we see that $U_n\-\{(\a,\dots,\a)\}$ is non-degenerate for any $\a\in\D$ when $|\D|\geq2$.
\end{example}
The above examples already show that there are non-degenerate relations of arbitrarily high arity. Therefore, the situation with Cartesian factorization is similar to factorization of polynomials over $\Q$ with irreducible polynomials of arbitrarily high degrees. The next theorem shows, by a counting argument, that, as far as reductive power is concerned, the situation is even worse -- `almost all' relations are non-degenerate.
\begin{theorem}\label{DegLim} The share of degenerate $n$-ary relations among all such relations on a domain $\D$ asymptotically vanishes when $|\D|\geq2$ and $n\to\infty$, or when $n\geq2$ and $|\D|\to\infty$. 
\end{theorem}
\begin{proof} The set of all $n$-ary relations on $\D$ is $\P(\D^n)$, so their total number is $|\P(\D^n)|=2^{|\D|^n}$. For any degenerate relation there is a non-empty proper subset $\L\subset\N_n$ such that it is a Cartesian product over the bipartition $\L\cup\L^c$, and every bipartition is counted twice by $\L$ due to the symmetry between $\L$ and $\L^c$. There are $\binom{n}{k}$ subsets of $\N_n$ with $|\L|=k$, and we can assign $2^{|\D|^{|\L|}}=2^{|\D|^k}$ and $2^{|\D|^{|\L^c|}}=2^{|\D|^{n-k}}$ different relations to each of the Cartesian factors. Of course, some of the products obtained in this way may coincide, so we only get an upper bound for the total number $N_{deg}$ of degenerate relations:
\begin{equation*}
N_{deg}\leq\frac12\,\sum_{k=1}^{n-1}\binom{n}{k}\,2^{|\D|^k+|\D|^{n-k}}.
\end{equation*} 
By calculus, the function $|\D|^k+|\D|^{n-k}$ on $1\leq k\leq n-1$ takes its maximal values at the endpoints when $|\D|\geq2$, so $|\D|^k+|\D|^{n-k}\leq|\D|+|\D|^{n-1}$ for all $k$. Therefore,
\begin{equation}\label{Ndeg}
N_{deg}\leq2^{|\D|+|\D|^{n-1}-1}\sum_{k=1}^{n-1}\binom{n}{k}
<2^{n-1+|\D|+|\D|^{n-1}},
\end{equation}
because $\sum_{k=1}^{n-1}\binom{n}{k}=2^n-2<2^n$. The conclusions now follow by applying calculus to the fraction $\frac{2^{n-1+|\D|+|\D|^{n-1}}}{2^{|\D|^n}}$.
\end{proof}

\section{Joins and dependencies}\label{join}

As far as analysis of relations is concerned, Cartesian products do not take us very far. Asymptotically, in either domain size or arity, almost all relations do not factor into them. A natural move is to consider more general decompositions. In terms of predicates, Cartesian products correspond to free conjunctions, so let us allow conjunctions that are not necessarily free, i.e. some variables in them are identified. This leads to the notion of {\it natural join}, or join for short, introduced by Codd in the context of relational databases \cite{Codd70}. Its binary version, {\it Boolean product}, goes back to the 19th century \cite{Brun91}. Our definition of join parallels the definition of Cartesian product, the only difference is that the relation schemes no longer have to form a partition of $\S$, only a cover. 
\begin{definition}\label{JoinCov}
Given finite subsets $\L_i\subseteq\S$ and attributed relations $R_i\in\D^{\L_i}$, their (natural) join $R\in\D^{\,\cup_i\L_i}$ is defined as the set of concatenations of tuples from $R_i$ that match on overlaps of $\L_i$:
$$
R_1\Join\dots\Join R_m:=\{(a^1|\dots|\, a^m)\,\Big|\,a^i\in R_i,\,a^i_{\L_i\cap\L_j}\!=a^j_{\L_i\cap\L_j}\}.
$$
A relation $R\in\D^\S$ is a \textbf{(natural) join over a cover} $\S=\L_1\cup\dots\cup\L_m$ when there exist $R^{\L_i}\in\D^{\L_i}$ such that $R=R^{\L_1}\Join\dots\Join  R^{\L_m}$, and $R^{\L_i}$ are called its \textbf{join factors}. It is called \textbf{join reducible} when  $0<|\L_i|<|\S|$.
\end{definition}
\noindent As the notation indicates, join of multiple factors is obtained by iterating Boolean product $\Join$. On attributed relations Boolean product is commutative and associative, just like Cartesian product, because ordering of positions does not get in the way. 

In terms of predicates, joins are just conjunctions of predicates over subsets of variables with some variables shared among them:
\begin{equation}\label{JoinConj}
R(x_1,\dots,x_n)=R^{\L_1}(x_{\L_1})\land\dots\land R^{\L_m}(x_{\L_m}).
\end{equation} 
This formula looks exactly like the free conjunction \eqref{FreeConj}, except now $\L_i$ may overlap, e.g. $R^{1,3}(x_1,x_3)\land R^{1,2,4}(x_1,x_2,x_4)\land R^{3,4}(x_3,x_4)$ is a join. In mathematical logic,  joins of predicates are called quantifier-free primitive positive formulas, and join closures are called weak partial co-clones in CSP literature \cite{JLNZ13,Lag17}.

Note that \eqref{JoinConj} no longer implies that $R^{\L_i}=\pi_{\L_i}R$, as it did for Cartesian products. This is because some tuples in $R^{\L_i}$ may not have companion tuples in other $R^{\L_j}$ consistent with them on overlaps, and do not make it into the join. However, \eqref{JoinConj} remains valid if we {\it replace} $R^{\L_i}$ by $\pi_{\L_i}R$ because all and only joinable tuples are present in the projections. In this weakened sense, joins share the constructive property of Cartesian products -- their (minimized) factors can be recovered as projections of the join itself, as pointed out by Rissanen \cite{Riss77}. When the join factors contain only joinable tuples they are said to join {\it completely} \cite[2.4]{Maier}.

There is also an analog of the independence criterion for Cartesian products \cite{Lee83}.
\begin{theorem}[\textbf{Join criterion}]\label{JoinInd} 
A relation $R\subseteq\D^\S$ is a join over a cover $\S=\L_1\cup\dots\cup\L_m$ if and only if the values of its tuples on $\L_i$ can be chosen independently up to consistency on overlaps, i.e. for any collection of $\a^i\in\pi_{\L_i}R$ with $\a^i_{\L_i\cap\L_j}\!=\a^j_{\L_i\cap\L_j}$ there is common $a\in R$ such that $a_{\L_i}=\a^i$.
\end{theorem}
\noindent The proof is analogous to the proof of Theorem \ref{CartInd} and we omit it. 

Unfortunately, this criterion is not nearly as useful in detecting either join reducibility or irreducibility. More constructive reducibility conditions are usually formulated in terms of {\it data dependencies} \cite{BeeFag81,Fag77,Lee83,Maier,MenMai}, the simplest of which is functional.
\begin{definition}\label{FunDep}
Let $\L,\M\subseteq\S$. We say that a relation $R\subseteq\D^\S$ has a functional dependency $\L\to\M$ when its $\L$ values determine its $M$ values, i.e. for any $a,b\in R$ if $a_\L=b_\L$ then $a_\M=b_\M$. When $\K\to\K^c$ for some $\K\subset\S$ then $\K$ is called a \textbf{key} of $R$, and more precisely a $\boldsymbol{k}\textbf{-key}$, where $k=|\K|$.
\end{definition}
\noindent In databases, key column(s) are those whose entries determine the entire record, for example, the ID column. This situation was prototypical for choosing the terminology, the key column(s) code objects the relation is about, and the rest list their attributes. What matters to us is that availability of keys ensures join reducibility: instead of relating all attributes to their objects in a single relation we can relate them one at a time. The idea goes back to the work of Peirce and is in close affinity to his ``hypostatic abstraction" (see \cite{Bur97,Kosh22} and Section \ref{projoin}).
\begin{theorem}\label{KeyRed} 
Any $R\subseteq\D^\S$ with a $k$-key $\K=\{i_1,\dots,i_k\}\subseteq\S$ decomposes into a join of $(k+1)$-aries as 
\begin{equation}\label{KeyRedRel}
R=\join\limits_{i\not\in\K}\pi_{\K\cup\{i\}}R
=R_{i_{k+1}}\!\Join\dots\Join R_{i_{n}},
\end{equation}
where $R_i:=\pi_{\K\cup\{i\}}R$. In the predicate form,
\begin{equation}\label{KeyRedPred}
R(x_1,\dots,x_n)=R_{i_{k+1}}(x_{i_1},\dots,x_{i_k}, x_{i_{k+1}})\land\dots\land R_{i_{n}}(x_{i_1},\dots,x_{i_k}, x_{i_n})
\end{equation}
In particular, $R$ is join reducible when $k\leq n-2$, and join reducible to binaries when $k=1$.
\end{theorem}
\begin{proof} Without loss of generality, we may rename the attributes into integers and assume $\S=\N_n$, $\K=\{1,\dots,k\}$. Suppose $a\in R$, then for any $i\geq k+1$ the subtuple $(a_1,\dots,a_k,a_i)\in R_i=\pi_{\K\cup\{i\}}R$ because it is the projection of $a$ to $\K\cup\{i\}$. Hence $a\in R_{k+1}\!\Join\dots\Join R_{n}$.

Now suppose $a\in R_{k+1}\!\Join\dots\Join R_{n}$, i.e. $(a_1,\dots,a_k,a_i)\in R_i$ for all $i\geq k+1$. By definition of $R_i$, there must be tuples $b^i\in R$ with $b^i_j=a_j$ for $j\leq k$ and $b_i^i=a_i$. But the first $k$ positions are a key, and all $b^i$ coincide on them with $a$, therefore, $a=b^1\in R$.
\end{proof}
\begin{example}\label{ExKeyRed} In the identity relation $I_n$ every position is a $1$-key. Taking $\K=\{1\}$ and applying \eqref{KeyRedPred}, we get the familiar reduction to binaries
\begin{equation*}
I_n(x_1,\dots,x_n)=I_2(x_1,x_n)\land\dots\land I_2(x_{n-1},x_n)
=\,(x_1=x_n)\land\dots\land(x_{n-1}=x_n).
\end{equation*}
\end{example}
\begin{example}
The division with remainder quaternary $\textup{Div}(n,d,q,r)$ ($n=dq+r$ with $n$ the dividend, $d\neq0$ the divisor, $q$ the quotient and $0\leq r\leq d-1$ the remainder) on $\N\cup\{0\}$ has a $2$-key consisting of the dividend and the divisor. Therefore, it reduces to the join of two ternaries: $\mathrm{Div}^{1,2,3}(n,d,q)=\,\big(q=\lfloor\frac{n}{d}\rfloor\big)$, where $\lfloor\cdot\rfloor$ is the floor function, and $\textup{Div}^{1,2,4}(n,d,r)=\big(n\equiv r\!\pmod d\big)$.
\end{example}

Functional dependency is not necessary for join reducibility, not even by join decompositions of the special form \eqref{KeyRedRel}-\eqref{KeyRedPred}. This is fortuitous because it is a rare occurrence, just like its polar opposite, Cartesian independence of values, and we could not join reduce many relations if it was required. 
\begin{example} Consider the ternary $R$ on $\D=\{\a,\b\}$ given by the table below.
\medskip

$R$ = \begin{tabular}{|C|C|C|}
\hline
 \a & \a & \a \\  \hline
 \a & \a & \b \\ \hline
 \a & \b & \a \\ \hline
 \a & \b & \b \\  \hline
 \b & \a & \b \\ \hline
\end{tabular}\ \ \ \ \ \ \ \ \ \ \ \ \
$R^{1,2}$ = \begin{tabular}{|C|C|}
\hline
 \a & \a \\  \hline
 \a & \b \\ \hline
 \b & \a \\ \hline
 \end{tabular}\ \ \ \ 
 $R^{1,3}$ = \begin{tabular}{|C|C|}
\hline
 \a & \a \\  \hline
 \a & \b \\ \hline
 \b & \b \\ \hline
 \end{tabular}
 \medskip
 
\noindent Since $\a$ is repeated in the first column the latter cannot be a key. Nonetheless, $R$ is a join of the form \eqref{KeyRedPred}: $R(x,y,z)=R^{1,2}(x,y)\land R^{1,3}(x,z)$, with the factors as shown above. This is because all possible combinations are present in $R$ in the second and third position after $\a$. As a result, $\s_{x_1=\a}R$ is a Cartesian product of unaries, and so is $\s_{x_1=\b}R$ over the {\it same} partition (trivially, as a singleton).
\end{example}
\noindent The above example illustrates a more general dependency introduced by Delobel, Fagin and Zaniolo around 1977 \cite[7.11]{Maier}, the {\it multivalued dependency}.
\begin{definition}\label{MultDep}
Let $\S=\M\cup\L_1\cup\dots\cup\L_m$ be a partition. We say that an $R\subseteq\D^\S$ has a multivalued dependency  $\M\twoheadrightarrow\L_1\cup\dots\cup\L_m$ when for all $\a\in\pi_\M R$ the selections $\s_{x_\M=\a}R$ are Cartesian products over the common partition $\L_1\cup\dots\cup\L_m$. When $\L_i$ are singletons for all $i$ then $\K:=\M$ is called a \textbf{multikey} of $R$, and, more precisely, a $\boldsymbol{k}\textbf{-multikey}$ when $k=|\K|$.
\end{definition}
\noindent Note that both functional dependency and Cartesian independence (with $\M=\emptyset$) are special cases of multivalued dependency. It turns out that such dependency is both necessary and sufficient for join decompositions of the special form \eqref{MultKeyRedRel}-\eqref{MultKeyRedPred}, when the same attributes are shared by all factors.
\begin{theorem}[\textbf{Fagin \cite{Fag77}}]\label{MultKeyRed} 
Let $\S=\M\cup\L_1\cup\dots\cup\L_m$ be a partition. An $R\subseteq\D^\S$ has a join decomposition of the form
\begin{equation}\label{MultKeyRedRel}
R=\join\limits_{i=1}^m\pi_{\M\cup\L_i}R
=R_1\!\Join\dots\Join R_m,
\end{equation}
where $R_i:=\pi_{\M\cup\L_i}R$, or, equivalently, 
\begin{equation}\label{MultKeyRedPred}
R(x_1,\dots,x_n)=R_1(x_\M\,|\,x_{\L_1})\land\dots\land R_m(x_\M\,|\,x_{\L_m}),
\end{equation}
if and only if $\M\twoheadrightarrow\L_1\cup\dots\cup\L_m$ is a multivalued dependency in $R$.
\end{theorem}
\begin{proof} One direction is trivial, if $R$ is of the form \eqref{MultKeyRedPred} then substituting $x_\M=\a$ turns it into a free conjunction since $\L_i$ are disjoint. For the other direction, let $\L:=\L_1\cup\dots\cup\L_m$ and note that for any $\a\in\pi_\M R$ we have 
\begin{equation}
\s_{x_\M=\a}R\,(x_\L)=R_1^\a(x_{\L_1})\land\dots\land R_m^\a(x_{\L_m}),
\end{equation}
with some $R_i^\a\subseteq\D^{\L_i}$, by definition of Cartesian product over a partition. It remains to set 
$R_i(a):=R_i^{a_\M}(a_{\L_i})$ when $a_\M\in\pi_\M R$ and $0$ (false) otherwise to get \eqref{MultKeyRedPred}.
\end{proof}
Multivalued dependency reductions are popular in database design due to their constructive nature, but they do not exhaust all possible join reductions. For example, the diversity relation can be join reduced even to binaries, but not in the form \eqref{MultKeyRedPred}:
\begin{equation}\label{DiverseJoin}
D_n(x_1,\dots,x_n)=\bigwedge\limits_{i<j} I_2(x_i,x_j)=\bigwedge\limits_{i<j} (x_i\neq x_j).
\end{equation}
However, the more complex a dependency the harder it is to detect and use to produce reductions. The definition of join dependency, for example, amounts to just saying that the relation is a join (but see \cite{MenMai} for a somewhat more cogent characterization).

Be it as it may, we will now show that even taking all possible join reductions into account there are still irreducible relations of arbitrarily high arities. To this end, we observe that the proof of Theorem \ref{DegTests} did not use the fact that $\L_i$ form a partition, and so goes through for join reductions without a change. Therefore, its conclusion can be strengthened.
\begin{theorem}\label{JoinIrredTests} Suppose one of the following conditions holds for $R\subseteq\D^\S$.

\textup{(i)} $R$ is not universal, but for any non-empty proper subset $\emptyset\subset\L\subset\S$ \par we have $\pi_{\L}R=\D^{\L}$. 

\textup{(ii)} $\neg R$ is not empty, but for any $i\in\S$ we have $\pi_i(\neg R)\neq\D$.

Then $R$ is join irreducible.
\end{theorem}
\noindent The reasoning from Example \ref{DegNotIn} now shows that $\neg I_n$ and $\neg I_n^{\A}$ for non-empty proper $\A\subset\N_n$ are join irreducible for any $n\geq1$ and $|\D|\geq2$. Worse yet, join reductions leave asymptotically `almost all' relations irreducible, albeit not in as strong a sense as Cartesian factorizations.
\begin{theorem}\label{JoinRedLim} The share of join reducible $n$-ary relations among all such relations on a domain $\D$ is $<1$ for $|\D|>n$, and asymptotically vanishes when $n\geq1$ and $|\D|\to\infty$. 
\end{theorem}
\begin{proof} 
Suppose $R$ is join reducible to a conjunction \eqref{JoinConj}. We can always interpret $R^{\L_i}$ as having attributes from any superset $\L_i'\supset\L_i$ by conjoining it with the universal relation on the scheme $\L_i'\-\L_i$. Since all $\L_i$ are proper subsets of $\N_n$ each of them is contained in one of $\L\subset\N_n$ with $|\L|=n-1$. Therefore, after merging predicates with identical schemes if necessary, we can represent $R$ in the form (see \cite{MenMai})
\begin{equation*}
R(x_1,\dots,x_n)=\bigwedge\limits_{|\L|=n-1}\!\!\!\!\! R^{\L}(x_\L)=\bigwedge\limits_{i=1}^n R^{\N_n\-\{i\}}\big(x_{\N_n\-\{i\}}\big).
\end{equation*}
There are $n$ such $\L$ and $|\P(\D^\L)|=2^{|\D|^{|\L|}}=2^{|\D|^{n-1}}$\!\! choices for each $R^{\L}$. Some of them may be incompatible, so we get an upper bound on the number of join reducible relations: $N_{jred}\leq2^{n|\D|^{n-1}}$. The conclusions now follow by applying calculus to the fraction 
\begin{equation}\label{Njred:Ntot}
\frac{2^{n|\D|^{n-1}}}{2^{|\D|^n}}=2^{-(|\D|-n)|\D|^{n-1}}.
\end{equation}
\end{proof}
\noindent Note that if $|\D|$ is fixed and $n\to\infty$ the ratio in \eqref{Njred:Ntot} goes to $\infty$, not to $0$, so we cannot conclude that increasing arity on a fixed domain also leads to a vanishing share of join reducible relations, as we could for degenerate ones in Theorem \ref{DegLim}. 

\section{Projective joins and hypostatic abstraction}\label{projoin}

As noted by Rissanen \cite{Riss77}, joins are the most general form of decomposition where the constituents can be recovered by taking projections. Yet they still admit irreducibles of arbitrarily high arity, and many of them. If we want a more manageable collection of irreducibles we need to relax the constructivism. There is also a more practical reason. When relations are entered into databases the expected data dependencies are often prescribed by design. What if potential entrants do not have them? They are then ``pre-treated" before incorporation to mend that. For example, full names are expected to function as relation keys, but identical namesakes do occur sometimes. This is mended by attaching ID columns to the tables that restore  uniqueness (and hence, functional dependency).

Such pre-treatment is often done ``under the desk", but it is of interest to develop its devices systematically. Attaching a key appears already in C.S. Peirce's works on relations from 1890s under the name of ``hypostatic abstraction" \cite{Bur91,Bur97}. Consider the ternary relation $G(x,y,z):=$``$x$ gives $y$ to $z$". To reduce it to binaries, Peirce introduces (``hypostatizes") new abstracted objects $t$, the acts of giving, and treats $x$, $y$ and $z$ as their attributes, the giver, the gift, and the recipient. If we denote their relations to the object by $G'(t,x)$, $G''(t,y)$ and $G'''(t,z)$, respectively, then we can form a key-form join $G'(t,x)\land G''(t,y)\land G'''(t,z)$ that carries all the information of $G(x,y,z)$. Except it also features the acts of giving $t$ that are not among the original attributes. What we really need to say is that {\it there is} such an act for given $x$, $y$ and $z$:
\beq\label{give}
G(x,y,z)=\exists t\left[G'(t,x)\land G''(t,y)\land G'''(t,z)\right].
\eeq
This is a reduction of sorts, but it is no longer a join, rather a projection of a join. Its factors cannot be recovered from $G$ even if we know their relation schemes because of the freedom in $t$ assignments. But it does fit well with the algebraic ideology of joins: first, we represent a given $R$ as a projection of a higher arity relation $\hR$, and then factor $\hR$ into a join. While the factors are not projections of $R$, they are projections of $\hR$, as is $R$ itself.
\begin{definition}\label{ProJoinDef}
The \textbf{projective join}, or just \textbf{projoin}, of attributed relations $R_i\subseteq\D^{\L_i}$ with the set of projected attributes $\Gamma\subseteq\cup_i\L_i$ is $\pi_\Gamma\left[R_1\Join\dots\Join R_m\right]$, i.e. the projection of their (natural) join to $\Gamma$.
A relation $R\subseteq\D^\S$ is a \textbf{projoin over a cover} $\T\cup\S=\L_1\cup\dots\cup\L_m$ with $\T\cap\S=\emptyset$ when there exist $R^{\L_i}\subseteq\D^{\L_i}$, called its \textbf{projoin factors}, such that 
$$
R=\pi_{\S}\!\left[R^{\L_1}\Join\dots\Join R^{\L_m}\right]=\pi_{\S}\hR.
$$
The join $\hR$ is called the \textbf{augmented relation}, and elements of $\T$ (projoin) \textbf{parameters}. $R$ is called \textbf{projoin reducible} when it is a projoin with $0<|\L_i|<|\S|$.
\end{definition}
\noindent The corresponding predicate formula for projoins is:
\begin{equation}\label{ProJoinConj}
R(x_1,\dots,x_n)=\exists x_\T\left[R^{\L_1}(x_{\L_1})\land\dots\land R^{\L_m}(x_{\L_m})\right].
\end{equation} 
It is more typical to consider reductions in the algebra with two separate operations, join and projection, and project-join expressions are studied in \cite{DunMik,YanPap} and are part of a link between the relational model and Tarski's cylindric algebras  \cite{Buss,ImLip}. In terms of predicates, passing from joins to projoins means that in reductions we allow existential quantification on top of conjunction and identification of variables, i.e. consider primitive positive expressions (assuming the binary identity is among the available predicates). Up to switching from attributed to ordinary relations, the projoin reduction is essentially equivalent to reduction in the relational (co-clone) algebra of clone theory \cite{CorPos04,Lau,PosKal}. As such, it plays a key role in characterizing complexity of constraint satisfaction problems (CSP), and the study of generating sets of small arity in connection with them can be applied to solving Peircian reduction problems on finite domains \cite{BRSV05}.

Just as join dependencies can be expressed using joins, generalized data dependencies can be expressed using project-join equations \cite{YanPap}. With projoins, one can use the flexibility of augmenting the relation to make the decomposition methods for joins more broadly applicable. The simplest method used a key, but whether a relation has a key is, in some ways, a bookkeeping matter. This is another reason for moving from joins to projoins.
\begin{definition}\label{KeyAdm}
We say that a relation $R\subseteq\D^n$ \textbf{admits a key (multikey)} when it is a projection of a relation $\hR\subseteq\D^{\T\cup\S}$ with a key (multikey), i.e. $R=\pi_{\S}\hR$. 
\end{definition}
\noindent Admitting a key is less of a bookkeeping matter than already having it. Note that $\T$ need not be the key in itself, the key may combine $\T$ with some attributes from the original relation. However, whether $R$ admits a key or not does not depend on which attributes are used in it. 
\begin{lemma}\label{KeyAdmCrit} A relation $R\subseteq\D^\S$ admits a $k$-key if and only if $|R|\leq|\D|^k$.
\end{lemma}
\begin{proof} 
Suppose $R=\pi_\T\hR$ and $\K\subseteq\T\cup\S$ is a key of $\hR$ with $|\K|=k$. Since there are at most $|\D|^k$ distinct tuples in the $\K$ positions of $\hR$, and they determine the remaining values, $|\hR|\leq|\D|^k$. But projection never increases the cardinality of a relation, so $|R|=|\pi_{\S}\hR|\leq|\hR|\leq|\D|^k$.

Conversely, suppose $|R|\leq|\D|^k$. Then we can pick any $\T$ of cardinality $k$ with unused attributes, and assign to every tuple in $a\in R$ a unique label $t(a)\in \D^\T$. We define $\hR$ as the set of augmented tuples $\hR:=\{(t(a),a)\,|\,a\in R\}$. By construction, the first $k$ positions of $\hR$ are its key and $\pi_{\S}\hR=R$.
\end{proof}
\noindent Recall from the previous section that relations {\it with} a $1$-key are quite rare on finite domains. A striking consequence of this lemma is that on infinite domains, any relation {\it admits} a $1$-key. Indeed, by the cardinal arithmetic, $|D|^n=|\D|$ for any finite $n$, so $|R|\leq|\D|^n=|\D|$ for any $n$-ary relation on $\D$. The next theorem generalizes Peirce's trick \eqref{give} for the relation of giving and connects it to relation keys, see also \cite{Bur91} for the general case.
\begin{theorem}[\textbf{Hypostatic abstraction}]\label{ProKeyRed} 
Any $R\subseteq\D^\S$ with $|R|\leq|D|^k$ decomposes into a projoin of $(k+1)$-aries as 
\begin{equation}\label{ProKeyRedRel}
R=\pi_{\S}\!\left[\join\limits_{i=1}^n\pi_{\T\cup\{i\}}\hR\right]
=\pi_{\S}\!\left[R_1\!\Join\dots\Join R_n\right],
\end{equation}
where $R_i:=\pi_{\T\cup\{i\}}\hR$ for some $\hR\subseteq\D^{\T\cup\S}$ with $|\T|=k$. In the predicate form,
\begin{equation}\label{ProKeyRedPred}
R(x_1,\dots,x_n)=\exists t_1\dots\exists t_k\left[R_1(t_1,\dots,t_k, x_1)\land\dots\land R_n(t_1,\dots,t_k, x_n)\right].
\end{equation}
In particular, $R$ is projoin reducible when $k\leq n-2$, and projoin reducible to binaries when $k=1$.
\end{theorem}
\begin{proof} The augmented relation $\hR\subseteq\D^{\T\cup\S}$ is constructed in the proof of Lemma \ref{KeyAdmCrit}. Since the augmented attributes $\T$ are its $k$-key, by construction, the decomposition formula \eqref{ProKeyRedRel} follows from Theorem \ref{KeyRed}. In \eqref{ProKeyRedPred} we specialized to $\T=\{t_1,\dots,t_k\}$ for concreteness, and ordered $\T\cup\S$ so that the $\T$ positions go before the original ones.
\end{proof}
\noindent Hypostatic abstraction radically simplifies the picture of reducibility on infinite domains, and delivers a ``small" set of irreducibles that eluded us with Cartesian products and joins. As we mentioned, this is similar to factorization of polynomials over $\R$, where the only irreducible ones are linear and quadratic. 
\begin{corollary}\label{InfDomProRed} 
Any $n$-ary with $n\geq3$ on an infinite domain reduces to a projoin of $n$ binaries. The only projoin irreducible relations on such domains are all unaries and non-degenerate binaries.
\end{corollary}
\begin{proof} The first claim is a direct consequence of cardinal arithmetic for infinite $|\D|$ and Theorem \ref{ProKeyRed}. Unaries have nothing to be reduced to. Binaries can only be reduced to unaries. But quantifying over a conjunction of unary predicates still produces a conjunction of unary predicates. Since unaries have only one variable each it is a free conjunction, i.e. an unaric Cartesian product. Thus, projoin reducible binaries must be degenerate.
\end{proof}
\noindent L\"owenheim might have been the first to state and prove this result in a more formal manner in 1915 \cite{Low}. His proof used somewhat more cumbersome iterated pairing construction instead of hypostatic abstraction. In fact, the result is even stronger than stated. Recall that projoins are ranked by the number of parameters, with joins having none, $1$-key hypostatic reductions like \eqref{give} having one, and so on. It follows from the proof that all higher arity relations are not only projoin reducible on infinite domains, but even projoin reducible with a single parameter. This is not the case on finite domains, as we will now demonstrate.
\begin{example}\label{NotI3ProIrred} Consider $\neg I_n$ on a finite domain $\D$. If it is projoin reducible then factors without augmented attributes can be dropped. Indeed, in the predicate representation \eqref{ProJoinConj} those factors can be taken out of the scope of quantifiers, and each carries a proper subset of $\neg I_n$'s variables. But $\neg I_n$ is universal on any proper subset of its positions, and hence so must be those factors since they enter conjunctively. Lower arity factors can be absorbed into those of arity $n-1$, so for $n=3$ any projoin reduction with one parameter condenses to just this
\beq\label{NotI3ProJoin}
\neg I_3(x_1,x_2,x_3)=\exists t\left[R_1(t,x_1)\land R_2(t,x_2)\land R_3(t,x_3)\right].
\eeq

We will now restrict to a two-element domain $\D:=\{\a,\b\}$. Then $t$ can only take two values, and, recalling the interpretation of the existential quantifier, \eqref{NotI3ProJoin} represents $\neg I_3$ as a disjunction of two unary conjunctions, i.e. a union of two unary Cartesian products. Possible cardinalities of such products on a two-element domain are $1,2,4$ and $8$. But $8$ is impossible because $|\neg I_3|=6<8$, and $4$ comes from a product of two doublets and a singleton. But a product of two doublets will include both $(\a,\a)$ and $(\b,\b)$, so multiplying it by any singleton will produce a constant tuple not in $\neg I_3$. So $4$ is also impossible. As for $1$ and $2$, no pair of them can cover all $6$ tuples of $\neg I_3$. 

Thus, on two-element domains $\neg I_3$ is not projoin reducible with only one parameter. A similar counting argument works also for $|\D|=3$. For larger $|\D|$ combinatorial considerations of this sort quickly become unmanageable. For $n>3$ even single parameter projoin reductions are no longer of the simple form \eqref{NotI3ProJoin}, as free variables can be shared and $\neg I_n$ need not be a union of Cartesian products, see Example \ref{TriPro2Irred}.
\end{example}
A counting argument similar to that of Theorem \ref{JoinRedLim} further shows that the reducibility behavior on finite domains is quite different.
\begin{theorem}\label{ProJoinRedLim} The share of $n$-ary, $n\geq3$, relations projoin reducible with $k$ parameters among all such relations on a domain $\D$ is $<1$ for $|\D|>\binom{n+k}{n-1}$, and asymptotically vanishes when $|\D|\to\infty$. 
\end{theorem}
\begin{proof} As in the proof of Theorem \ref{JoinRedLim}, we transform the projoin into a form with factors of maximal possible arity in a reduction, $n-1$. Only this time, due to the augmented positions, their total number is $\binom{n+k}{n-1}$ rather than $n$ (to which it reduces for $k=0$). The rest of the proof is analogous, and leads to $2^{-\left(|\D|-\binom{n+k}{n-1}\right)\,|\D|^{n-1}}$\!\! as the upper bound for the share.
\end{proof}
What are we to make of such a striking discrepancy between reducibilities on finite and infinite domains? The universal applicability of hypostatic abstraction on infinite domains derives directly from non-constructive and `unnatural' bijections between $\D$ and its Cartesian powers. By a theorem of Tarski \cite[11.3]{Jech}, $|\D|^2=|\D|$ for all infinite $\D$ is equivalent to the axiom of choice. Perhaps, it is of interest to consider reduction in models of set theory where availability of such bijections is restricted and the $|\D|,|\D|^2,\dots$ hierarchy does not collapse, such as ZF models with Dedekind finite sets, or to restrict maps allowed in definitions of reducing relations directly. These may model relations on large finite domains better than do all relations on infinite domains of ZFC.

\section{Projections as unions}\label{finprojoin}

As the results of the previous section show, reductions more general than \eqref{ProKeyRedRel}-\eqref{ProKeyRedPred} are still of interest on finite domains. Indeed, they may be of interest even on infinite domains as analyses of relations alternative to hypostatic abstraction. In this section we build on the relationship between existential quantification (projection) and unions exploited in Example \ref{NotI3ProIrred} to construct such reductions. It is more technical than the rest of the paper and can be skipped without loss of continuity. We start by giving a projective version of Fagin's theorem.
\begin{theorem}\label{ProMultKeyRed} 
An $R\subseteq\D^\S$ has a projoin decomposition of the form: \begin{equation}\label{ProMultKeyRedRel}
R=\pi_{\S}\!\left[\join\limits_{i=1}^m R_i\right]
=\pi_{\S}\!\left[R_1\!\Join\dots\Join R_m\right],
\end{equation}
where $R_i\subseteq\D^{\T\cup\L_i}$ for some partition $\S=\L_1\cup\dots\cup\L_m$, or, in the predicate form,
\begin{equation}\label{ProMultKeyRedPred}
R(x_1,\dots,x_n)=\exists t_1\dots\exists t_k\left[R_1(t_1,\dots,t_k, x_{\L_1})\land\dots\land R_n(t_1,\dots,t_k, x_{\L_m})\right],
\end{equation}
if and only if it is a union of Cartesian products over the common partition $\L_i$ with no more than $|\D|^{|\T|}$ terms.
\end{theorem}
\begin{proof} Suppose $R=\pi_{\S}\hR$ with $\hR=\join_{i=1}^m R_i$. Then, by Fagin's theorem, there is a multivalued dependency $\T\twoheadrightarrow\L_1\cup\dots\cup\L_m$ in $\hR$. But then, by definition, for all $a\in\pi_T\hR$ the selections $\s_{x_\T=a}R$ are Cartesian products over $\L_i$, and 
$$
R=\bigcup_{a\in\pi_T\hR}\s_{x_\T=a}R\,.
$$
Since $\pi_T\hR\subseteq\D^{\T}$ the number of terms in the union is no more than $|\D^{\T}|=|\D|^{|\T|}$.

Conversely, suppose $R$ is a union of Cartesian products over $\L_i$ with no more than $|\D|^{|\T|}$ terms. Select a distinct $a\in\D^{\T}$ for each term and label its term $R^a$. Then define
$$
\hR:=\bigcup_{a}\,\,\{(a,x)\,|\,x\in R^a\}.
$$
By construction, $R=\bigcup_{a}R^a$, and since $R^a$ are Cartesian products over $\L_i$ we have $\T\twoheadrightarrow\L_1\cup\dots\cup\L_m$. Now Fagin's theorem gives the desired decomposition of $\hR$.
\end{proof}
\noindent Theorem \ref{ProMultKeyRed} underscores a close relationship between projection and union, in predicate terms, between existential quantification and disjunction. One can represent existentially quantified formula as a disjunction by moving quantified variables into an index, as in $\exists tR(t,x)=\bigvee_tR^t(x)$. For finite relations, existential quantification can be converted into disjunctions completely. However, there are two limitations on such disjunctions. First, the number of disjuncts is limited by $|\D|^k$, where $k$ is the number of bound variables used, and second, the arity of disjuncts is then increased by $k$. As a result, if $k\geq n$ is needed to get the number of disjuncts under $|\D|^k$ then the disjunction does not convert into a projoin reduction.

On the other hand, allowing disjunctions/unions without restrictions completely trivializes the reduction problem. Any relation is a union of its tuples, and each tuple is the unary Cartesian product of singletons containing its members, i.e. $R=\bigcup_{a\in R}\,\{a_1\}\times\dots\times\{a_n\}$. Even if only finite unions are allowed every relation on a finite domain will `reduce' to unary Cartesian products. However, in principle, it may be of interest to explore unions with restrictions other than those imposed by the existential quantifier, e.g. with bounds on the number of terms independent of the size of the domain.

Example \ref{NotI3ProIrred} may suggest that non-identities $\neg I_n$ are projoin irreducible on finite domains. Indeed, $|\neg I_n|=|\D|^{n}-|\D|>|\D|^{n-2}$ for $n\geq3$ and $|\D|\geq2$, so hypostatic abstraction cannot  reduce them. We will now exploit the relationship between projections and unions to show that this is not the case when $\D$ is sufficiently large. But that requires somewhat more general projoins than Fagin-type decompositions \eqref{ProMultKeyRedRel}-\eqref{ProMultKeyRedPred}, namely, unions of joins that are not Cartesian products.
\begin{example}\label{NotInProRed} Recall the join reduction of $I_n$ from Example \ref{ExKeyRed}. Negating and applying de Morgan's law, we get 
\begin{equation*}
\neg I_n(x_1,\dots,x_n)=\neg I_2(x_1,x_n)\lor\dots\lor \neg I_2(x_{n-1},x_n)
=\bigvee_{j=1}^{n-1}\neg I_2(x_j,x_n).
\end{equation*}
We cannot convert this disjunction into an existentially quantified formula because the attribute sets in each disjunct are different, but we can combine all of them into a single cover $\L_i:=\{i,n\}$ for $i=1,\dots,n-1$. Then we define 
$$
R_i^j(x,y):=\begin{cases}\neg I_2(x,y),\,i=j\\
1,\,i\neq j,\end{cases}
$$
so that 
$$
\bigwedge_{i=1}^{n-1}R_i^j(x_i,x_n)=R_j^j(x_j,x_n)=\neg I_2(x_j,x_n),
$$
and
$$
\neg I_n(x_1,\dots,x_n)=\bigvee_{j=1}^{n-1}\bigwedge_{i=1}^{n-1}R_i^j(x_i,x_n).
$$
This disjunction is already amenable to ``existentialization". Assign a distinct $\a(j)\in\D$ to each disjunct, which requires $|\D|\geq n-1$, and set
$$
R_i(t,x,y):=\begin{cases}R_i^j(x,y),\,t=\a(j)\\
0,\,t\neq \a(j),\end{cases}
=\begin{cases}\neg I_2(x,y),\,t=\a(j),\,i=j\\
1,\,t=\a(j),\,i\neq j\\
0,\,t\neq \a(j).
\end{cases}
$$
Then the disjunction converts into a one parameter projoin of ternaries:
\beq\label{NotInProRedEq}
\neg I_n(x_1,\dots,x_n)=\exists t\left[\bigwedge_{i=1}^{n-1}R_i(t,x_i,x_n)\right].
\eeq
Thus, for $n\geq4$ and $|\D|\geq n-1$ the non-$n$-identity is projoin reducible with a single parameter. A similar construction works for non-$n$-diversity relation based on the join reduction \eqref{DiverseJoin}, but the condition is instead $|\D|\geq\frac{n^2-n}2$.

Going back to $\neg I_n$, the required size of the domain can be traded for arity. If we use pairs of elements to index the disjuncts then only $|\D|^2\geq n-1$ is required, and similarly $|\D|^k\geq n-1$ if we use $k$-tuples. But we must have $k+2\leq n-1$ so that it is still a reduction. Taking $k=n-3$ and noticing that $(n-1)^{\frac{1}{n-3}}\leq2$ for $n\geq5$ we conclude that $\neg I_n$ is reducible for $n\geq5$ on any $\D$ with $|\D|\geq2$ (albeit not necessarily to ternaries).
\end{example}
\noindent The construction in the above example can be generalized to prove the following theorem.
\begin{theorem}\label{ProNegJoin} 
Suppose an $n$-ary $R$ is join reducible to $N$ factors of arity at most $l$, and $N\leq|\D|^{n-l-1}$. Then $\neg R$ is projoin reducible. In particular, if $R$ has a $k$-key, with $k\leq n-4$ for $|\D|=2$ and $k\leq n-3$ for $|\D|\geq3$, then $\neg R$ is projoin reducible.
\end{theorem}
\begin{proof} As in Example \eqref{NotInProRed}, we present $\neg R$ as a disjunction and then convert it into a projoin. To get enough tuples for $N$ disjuncts we need $N\leq|\D|^k$, where $k$ is the number of projoin parameters, and $k+l\leq n-1$ so that the converted factors still have lower arity than $R$.

For the second claim, apply Theorem \ref{KeyRed} to obtain a join reduction of $R$ with $N=n-k$ and $l=k+1$, and note that $N\leq|\D|^{n-l-1}$ becomes $n-k\leq|\D|^{n-k-2}$. By calculus, $x\leq2^{x-2}$ for $x\geq4$ and $x\leq d^{x-2}$ for $x\geq3$ when $d\geq3$.
\end{proof}
The next example gives a taste of intricacies involved in ruling out general projoin reductions to prove unconditional projoin irreducibility.
\begin{example}\label{TriPro2Irred} Consider reducing $\neg I_3$ on a domain of size $|\D|=d$ with {\it two} parameters. As in Example \ref{NotI3ProIrred}, the ansatz reduces to
$$\label{NotI3Pro2Join}
\neg I_3(x_1,x_2,x_3)=\exists t_1\exists t_2\left[A(t_1,t_2)\land \bigwedge_{i=1}^{3}\Big(P^i(t_1,x_1)\land Q^i(t_2,x_2)\Big)\right].
$$
Replacing $t_1,t_2\in\D$ by indices $j,k\in\N_d$ we transform it into a disjunction
$$\label{NotI3Pro2Dis}
\neg I_3(x_1,x_2,x_3)=\bigvee_{j,k=1}^{d}a_{jk}\land \left[\bigwedge_{i=1}^{3}\Big(P^i_j(x_i)\land Q^i_k(x_i)\Big)\right].
$$
Since unaries represent subsets of $\D$, conjunctions with different $x_i$ their Cartesian products, and with the same $x_i$ their intersections we obtain a union of unary Cartesian products
$$\label{NotI3Pro2Set}
\neg I_3=\bigcup_{j,k=1}^{d}a_{jk} \left[\bigtimes_{i=1}^{3}\Big(P^i_j\cap Q^i_k\Big)\right],
$$
where $a_{jk}=0,1$ determines whether the term is included into the union. In contrast to Example \ref{NotI3ProIrred}, we can use up to $d^2$ terms in the union, but, as a tradeoff, the unary factors are not independent but must form a subset-valued rank $1$ Boolean matrix $R^i_{jk}=P^i_j\cap Q^i_k$ for each $i$. Not only do we have to check whether $\neg I_3$ splits into a union of up to $d^2$ unary Cartesian products, but also whether Boolean matrices $R^i_{jk}$ of their factors simultaneously factorize into outer products of Boolean vectors. 
\end{example}
\noindent The Boolean matrix factorization problem is quite involved even for $0$-$1$ matrices \cite{Miet21}, and with more parameters one would have to deal with simultaneous factorization of even more interdependent Boolean tensors. We can now observe the progression from a simple cardinality condition for key reductions (Theorem \ref{ProKeyRed}), to finding union decompositions with independent Cartesian factors for multikey reductions (Theorem \ref{ProMultKeyRed}), and to finding such decompositions with intricate tensor factorizations for general projoin reductions with many parameters. This should dispel the initial impression that with `enough' parameters every relation `clearly should be' projoin reducible. And it suggests a fruitful connection between reduction of relations and factorization of Boolean tensors, an active area of research in modern data science \cite{Miet11}.

\section{Bonds and teridentity}\label{bond}

In this section we will connect the theory of projoin reductions motivated by the database theory to the older theory of C.S. Peirce that supported his once controversial reduction thesis. While the main ideas are scattered in Peirce's writings, they did not gain currency until the formalizations by Herzberger \cite{Herz} and Burch \cite{Bur91}. For a modern mathematical approach see \cite{CorDau,CorPos04}.

Unlike Cartesian products and joins, projoin is not a single operation on attributed relations. Its definition additionally depends on the set of projected attributes. One may feel that this is too permissive and/or clumsy. A natural way to specify projected attributes intrinsically is to choose all and only those that are not shared by the factors, i.e. to project out the shared attributes. If we think of joining as selecting tuples that match on shared attributes and splicing them together then leaving out the matched parts can give a meaningful response to a query. 

Let us call such special projoins {\it pure}. Purity is a significant restriction on the operation: while the Fagin-type projoins \eqref{ProMultKeyRedRel}-\eqref{ProMultKeyRedPred} are pure, the reduction \eqref{NotInProRedEq} we constructed for $\neg I_n$ is not. In the case of two factors, the pure projoin is what Peirce called {\it relative product} (``relative" was his term for relation) \cite{Brun91}. For binary relations, functions in particular, it is simply their composition: $\exists y\left[P(x,y)\land Q(y,z)\right]$. Peirce considered the relative product more basic than joins and projections through which we defined it, and preferred to reverse the order of definitions.

Like Cartesian and Boolean products, the relative product is a commutative binary operation on attributed relations. This is because the positions with the quantified variable are determined by the shared attribute, not by the order of factors, so commutativity reflects the commutativity of conjunction. However, unlike the other two, the relative product is not associative. Consider binary relations $P(u,x), Q(u,y), R(u,z)$. Associating the first two first gives $\exists t\left[P(t,x)\land Q(t,y)\land R(u,z)\right]$, but the last two first gives $\exists t\left[P(u,x)\land Q(t,y)\land R(t,z)\right]$. And $\exists t\left[P(t,x)\land Q(t,y)\land R(t,z)\right]$ cannot be generated by relative products at all, so not even all pure projoins are generated. This is because identified variables (shared attributes) in joins remain free, and it does not matter whether we identify two of them at a time or more, we can repeat the exercise when iterating. But in the relative product identified variables are quantified over (projected out), and no new variable can be identified with them afterwards. 

In other words, in iterated relative products no attribute can be shared by more than two factors, and relative product is associative when restricted to triples of relations satisfying this condition. This is a further restriction on admissible projoins that restricts even relation schemes of factors that can appear in them. We will adopt Herzberger's term bond for this restricted class of projoins, although our bond is slightly more permissive than his along the lines adopted in \cite{CorPos04}.
\begin{definition}\label{BondDef}
A collection of attributed relations $R_i\subseteq\D^{\L_i}$ is called \textbf{bondable} when no three of $\L_i$ share an attribute. Their
\textbf{bond} is then the projective join with all the shared attributes projected out, i.e. the projected set is $\Gamma:=\cup_i\L_i\-\left(\cup_{i\neq j}\L_i\cap\L_j\right)$. A relation $R\in\D^n$ is a \textbf{bond over a cover} $\T\cup\S=\L_1\cup\dots\cup\L_m$ with $\T\cap\S=\emptyset$ when it is a bond of some $R^{\L_i}\subseteq\D^{\L_i}$, called its \textbf{bond factors}, with elements of $\T$ called its (bond) \textbf{parameters}. $R$ is called \textbf{bond reducible} when it is a bond with $0<|\L_i|<|\S|$.
\end{definition}
\noindent Note that we allow $\T=\emptyset$, and so, in contrast to Peirce and Herzberger, Cartesian products are bonds, with $0$ parameters. This makes bonds a generalization of Cartesian products alternative to joins, and simplifies some formulations. In terms of predicates, our definition means that the sets of free variables in different bond factors are disjoint, and every bound variable is present in exactly two factors. So the projoins $\exists t\left[P(x,t)\land Q(x,t)\right]$, $\exists t\left[P(x,t)\land Q(y,z)\right]$ are not bonds, and neither are hypostatic abstractions like \eqref{give}.

We will now show, following Peirce and Burch \cite{Bur91,Bur97}, that, somewhat surprisingly, this severely restricted operation leads to essentially the same notion of reducibility as general projoins. Peirce's first observation was that multiple variable identifications bond reduce to pairwise ones by using identity predicates, e.g.
\begin{multline}\label{multiidn}
R_1(t,x_{\L_1})\land\dots\land R_n(t,x_{\L_n})=\\
\exists t_1\dots\exists t_n\left[I_{n+1}(t,t_1,\dots,t_n)\land R_1(t_1,x_{\L_1})\land\dots\land R_n(t_n,x_{\L_n})\right].
\end{multline}
And his second observation was that $n$-identities $I_n$ for $n\geq4$ bond reduce to teridentities $I_3$: 
\begin{multline}\label{InviaI3}
I_{n}(x_1,\dots,x_n)=\\
\exists t_1\dots\exists t_{n-3}\left[I_3(x_1,x_2,t_1)\land I_3(t_1,x_3,t_2)\land\dots\land I_3(t_{n-3},x_{n-1},x_n)\right].
\end{multline}
Projoins with quantification into a single position can also be reduced to bonds by replacing the factors with the quantified variables by factors of lower arity, as in $\widetilde{P}(x_{\L}):=\exists t P(x_{\L},t)$. Applying the above identities converts any projoin into a bond of the original factors (up to renaming of attributes), their projections, and teridentities. Let us call this conversion {\it bond explication}. For example, the bond explication of $\exists\,t\left[P(t,x_1,x_2,t)\land\exists\,s\,Q(x_2,x_3,s,t)\right]$ is 
\vspace{-0.5em}
\begin{multline}\label{ExpliExm}
\exists\,t_1\exists\,t_2\exists\,t_3\,\exists\,y_1\exists\,y_2\big[I_{3}(t_1,t_2,t_3)\land I_{3}(y_1,x_2,y_2)\land P(t_1,x_1,y_1,t_2)\big.\\
\big.\land\exists\,s\left[I_1(s)\land Q(y_2,x_3,s,t_3)\right]\!\big].
\end{multline}
Since projection does not increase arity, and explication can add only ternaries, we have the following theorem.
\begin{theorem}\label{ProJointoBond} 
An $n$-ary with $n\geq4$ is projoin reducible if and only if it is bond reducible. A ternary is projoin reducible if and only if it decomposes into a bond of unaries, binaries and teridentities.
\end{theorem}
\begin{proof} The only claim not covered by bond explication is that ternaries that are bonds of unaries, binaries and teridentities are projoin reducible. Given such a bond, assign a new variable to each teridentity and replace by it all occurrences of the original variables from the teridentity. Then remove the teridentity and the quantifiers over its variables. By \eqref{multiidn}, this produces an equivalent expression, and, since all teridentities are removed, it is a projoin of unaries and binaries only.
\end{proof}
\noindent In particular, the projoin reduction  \eqref{NotInProRedEq} of $\neg I_n$ can be transformed not just into a pure one, but even into a bond reduction. Note that hypostatic abstraction \eqref{ProKeyRedPred} with $k=1$, when it applies, decomposes any ternary into a bond of binaries and teridentities. This gives us a bond analog of Corollary \ref{InfDomProRed}. 
\begin{theorem}[\textbf{Peirce's reduction thesis}]\label{InfDomBondRed} 
Any $n$-ary with $n\geq3$ on an infinite domain reduces to a bond of unaries, binaries and teridentities. The only bond irreducible relations on such domains are all unaries, non-degenerate binaries, and non-degenerate ternaries.
\end{theorem}
\noindent The first, reducibility, clause of the thesis is a direct consequence of Corollary \ref{InfDomProRed} and Theorem \ref{ProJointoBond}. Indeed, it is essentially equivalent to projoin reducibility to binaries alone. Ironically, L\"owenheim's result \cite{Low} to that effect once made it controversial due to the confusion between these closely related notions of reducibility \cite{Bur97,Kosh22}.

The second, irreducibility, clause, as applied to unaries and binaries, also follows trivially. Since bond is a special case of projoin, and they are projoin irreducible, they are all the more bond irreducible. However, the part concerning ternaries is non-trivial. It will follow from a graph-theoretic argument in Section \ref{ternPRT} (Theorem \ref{TerNonDeg}). 

While bond and projoin reducibilities are (almost) the same, bond reductions are much more special than general projoin reductions. Peirce felt that general projoin reductions conceal the complexity involved in attribute matching (variable identifications), and, as a result, do not provide ``true" or ``complete" analysis of a relation delivered by bonds \cite{Brun91,Kosh22}.

\section{Bonding diagrams}\label{diag}

In this section we introduce graphical representation of joins, projoins and bonds which pictures the structure of relations by analogy to diagrams of chemical decompositions of compounds into elements. It also allows to bring in graph-theoretic methods into analysis of reductions. The diagrams we describe are simplified versions of Peirce's existential graphs \cite{CorPos04} covering only a fragment of predicate logic (conjunction and existential quantifier) without the associated graphical calculus. We use some standard notation and terminology from graph theory \cite{Berge} throughout this and the following sections.
\begin{definition}\label{ProGraph}
The \textbf{projoin graph} is a labeled bipartite graph with vertices for each factor and each attribute of the projoin. An edge joins them when the attribute is in the relation scheme of the factor. The vertices are labeled by relation and attribute names, and sorted into \textbf{predicate vertices} (for factors), \textbf{free attribute vertices} (for projected attributes) and \textbf{bound attribute vertices} (for projected out attributes). The \textbf{valency} of a vertex is its graph-theoretic degree (the number of incident edges) for predicate and bound attribute vertices, and the graph-theoretic degree increased by $1$ for free attribute vertices.
\end{definition}
\begin{figure}[!ht]
\begin{centering}
a)\ \ \ \includegraphics[width=0.14\textwidth]{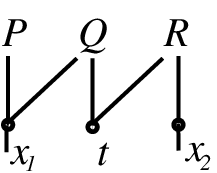} \hspace{0.4in} b)\ \ \ \includegraphics[width=0.18\textwidth]{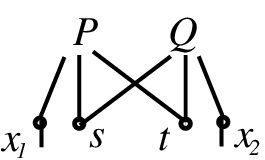} \par
\end{centering}
\caption{\label{ProGraphs} Projoin graphs for a) $\exists\,t[P(x_1)\land Q(x_1,t)\land R(t,x_2)]$; b) $\exists\,s\exists\,t[P(x_1,s,t)\land Q(s,t,x_2)]$.}
\end{figure}
\noindent In practice, we depict predicate vertices as predicate letters and attribute vertices as dots labeled by variables. This makes it easier to associate graphs to predicate formulas, and, since only matching of attributes matters in decompositions rather than their proper names, variable labels work well enough for our purposes. Free attribute vertices are additionally labeled by an extra stem (hanging edge) coming out of them, which explains the valency convention. Examples of projoin graphs and the corresponding predicate formulas are shown on Figure \ref{ProGraphs}.

Projoin graphs can get quite cluttered and are made somewhat more readable by converting them into bonding diagrams defined below. These diagrams are also better equipped to depict bonds.
\begin{definition}\label{BonDiag}
The \textbf{bonding diagram} is obtained from the projoin graph by replacing each bivalent bound vertex by an edge connecting the corresponding predicate vertices, and replacing each bivalent free vertex by a hanging edge from the corresponding predicate vertex. The new edges carry the attribute labels of the removed vertices, and the labels of free attribute vertices are moved to their stems. We call the remaining attribute vertices  of valency greater than $1$ \textbf{branch points}, and of valency $1$ \textbf{dead ends}. Hanging edges, incident to predicate vertices and branch points, are called \textbf{loose ends}. Bonding diagrams of bonds are called \textbf{bond diagrams}.
\end{definition}
\vspace{-1.5em}
\begin{figure}[!ht]
\begin{centering}
a)\ \ \ \includegraphics[width=0.26\textwidth]{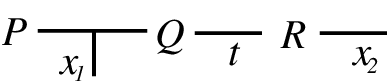} \hspace{0.1in} b)\ \ \ \includegraphics[width=0.24\textwidth]{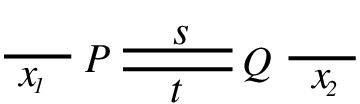} \par
\end{centering}
\caption{\label{BondGraphs} Bonding diagrams of projoins from Figure \ref{ProGraphs}.}
\end{figure}
\noindent Loose ends correspond to free variables, and dead ends to bound variables occurring in a single position. Branch points are the device for identifying variables in different predicates -- all edges attached to a branch point carry the same variable. The variable is free when one of the attached edges is a loose end, as in the T-shaped link on Figure \ref{BondGraphs}\,a), otherwise it is existentially quantified, as $t$ on Figure \ref{BondExpli}\,a). When a variable appears in only two predicates, no branch point is necessary, a simple edge connecting them suffices. Thus, bonding diagrams of bonds are graphically distinguished by having no branch points. When two or more variables appear in the same two predicates, the diagram displays a multiedge connecting their vertices, as in Figure \ref{BondGraphs}\,b). 

Note that if a bonding diagram has disconnected subdiagrams then the relations they represent are Cartesian factors of the original, except for the case when they have no free attributes, i.e. are closed formulas. If they are true all predicates in them can be dropped without any loss, and if false the factored relation is itself empty. From now on we will only consider projoins without such redundant predicates and call them {\it non-redundant}. 
\begin{corollary}\label{CartDiag} If a non-redundant bonding diagram of a relation is disconnected then the relation is degenerate. The connected components are bonding diagrams of its Cartesian factors. 
\end{corollary}
\noindent The converse is false for the trivial reason that one can use degenerate predicates in a reduction. But if a relation is degenerate it {\it admits} reductions with disconnected and non-redundant bonding diagrams.

We intentionally used a two-step definition instead of defining bonding diagrams directly to emphasize the singling out of bivalent vertices (pairwise attribute identifications), which highlights binary bonding, i.e. relative products. A bond diagram will have no attribute vertices, and all its attribute labels will attach to edges, including hanging edges. 

Our next observation is that bond explication also has a simple graphical interpretation in terms of bonding diagrams.
\begin{figure}[!ht]
\begin{centering}
a) \includegraphics[width=0.56\textwidth]{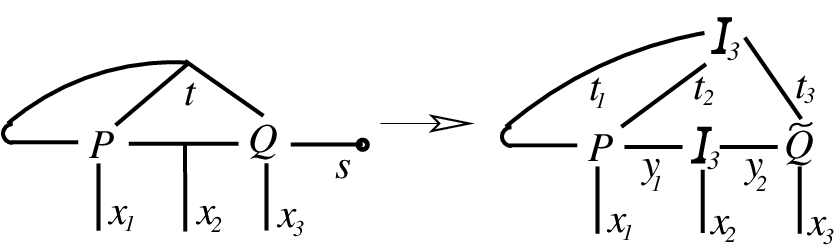} \hspace{0.05in} b)  \includegraphics[width=0.32\textwidth]{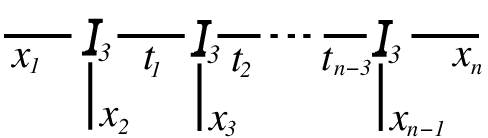} \par
\end{centering}
\caption{\label{BondExpli} a) Bond explication of $\exists\,t\left[P(t,x_1,x_2,t)\land\exists\,s\,Q(x_2,x_3,s,t)\right]$ as \eqref{ExpliExm}, here $\widetilde{Q}(t_3,y_2,x_3):=\exists s\,Q(t_3,y_2,x_3,s)$; b) bonding diagram of the reduction \eqref{InviaI3} of $n$-identity to teridentities.}
\end{figure}
\begin{corollary}\label{BondExpDiag} 
The bonding diagram of a bond explicated projoin is obtained from its original bonding diagram by absorbing dead ends into the adjacent predicate vertices, and replacing $n$-valent branch points by the diagrams of bond reductions of $n$-identities to teridentities (Figure \ref{BondExpli}).
\end{corollary}
\noindent It is particularly pronounced in the diagrams that attribute identifications (branch points) function like hidden factors in projoin reductions. Indeed, nothing substantive distinguishes them from predicate vertices in assembling the relation. Thus, one can see bond explication as analogous to adjoining ``ideal elements" ($I_n$) to uniformize factorizations in ring algebra.

\section{Complete reductions and ternarity}\label{tern}

So far we considered only general reductions, not complete reductions down to irreducibles.  In this section we will start looking at their structure of their complete reductions, but,
in the light of Peirce's reduction thesis, we consider only relations reducible to unaries, binaries and ternaries. They form a bond subalgebra of all relations, which is of interest even if irreducible higher arity relations exist on finite domains.

One can see from bonding diagrams that much work at putting a relation together is done by branch points. Hypostatic abstraction, for example, has a single branch point that holds together an otherwise loose collection of binaries. Bond explication removes branch points, but at a price of adding ternaries to the reduction. This suggests that ternaries play the role of relays in information exchange among the attributes, and their number quantifies the `complexity of relating attributes'. Mutual information between attributes has been studied in the context of database theory \cite{Mal}, and, more recently, as a measure of information integration in biological systems \cite{Teg16}. Another potential application is to designing conceptual schemas of databases friendly to natural language and human representation of knowledge and reasoning, as in Sowa's conceptual graphs that are based on bonding diagrams \cite{Cao10}. For more motivation and further discussion we refer to \cite{Kosh22}.

Thus, we will be interested in counting the number of ternaries in complete reductions of a relation, assuming that it is reducible to unaries, binaries and ternaries. Of course, this number may vary from one reduction to another, and what we really want is the {\it minimal} number over all possible reductions. 
\begin{definition}\label{ter+}
A bond is called \textbf{subternaric} when all of its factors have arity at most $3$. \textbf{Ternarity} of a relation, denoted $\ter$, is the minimal number of ternaries in its subternaric bond reductions, and $\infty$ if no such reductions exist. Non-redundant bond reductions with $\ter $ ternaries will be called \textbf{minimal bond reductions}.
\end{definition}
\noindent Ternarity of unaries and binaries is obviously $0$, and of ternaries is at most $1$. By the reducibility clause of Peirce's reduction thesis (Theorem \ref{InfDomBondRed}), all relations on infinite domains have subternaric bond reductions. Whether this holds on finite domains, i.e. whether $\ter (R)<\infty$ for all $R$, is an open problem, equivalent to projoin reducibility of all relations to binaries by Theorem \ref{ProJointoBond}. 

The following lemma is a direct consequence of the definitions.
\begin{lemma}\label{ter+prop} Ternarity is subadditive on relative products. For a projoin $R$ with factors $R_i$, free attributes indexed by $j$ with the $j$-th shared by $m_j$ factors, and bound attributes indexed by $k$ with the $k$-th shared by $n_k$ factors,
\beq\label{Proter+}
\ter(R)\leq\sum_i\ter(R_i)+\sum_j(m_j-1)+\sum_k(n_k-2).
\eeq
\end{lemma}
\begin{proof} 
Subadditivity is obvious because bonding two bond diagrams does not add predicate vertices or branch points. In the projoin graph of $R$, aside from $R_i$ vertices, we have free branch points of valencies $m_j+1$, and bound branch points of valencies $n_k$. According to \eqref{multiidn}-\eqref{InviaI3}, bond explication replaces the former with $m_j+1-2$ teridentities and the latter with $n_k-2$ teridentities, hence the sum in \eqref{Proter+}.
\end{proof}
Bounds on ternarity from above can be obtained from constructions of bond reductions. Recall that any projoin reduction can be explicated into bond reduction, and one can obtain projoin reductions by using keys (Theorem \ref{ProKeyRed}). However, only $1$- or $2$-keys produce subternaric reductions because for a $k$-key the factors have arity $k+1$.
\begin{theorem}\label{TerKeyRed} 
If an $n$-ary $R$ admits a $1$-key then $\ter(R)\leq n-2$, if it admits a $2$-key then $\ter(R)\leq 3n-4$, and if it already has a $2$-key then $\ter(R)\leq 3n-8$.
\end{theorem}
\begin{proof} Hypostatic abstraction \eqref{ProKeyRedPred} with $k=1$ reduces $R$ to a projoin of $n$ binaries with a single shared attribute, which is bound and shared by all $n$ factors. Therefore, the first two terms in \eqref{Proter+} vanish and the last one produces $n-2$. In the case of $k=2$ we obtain a projoin of $n$ ternaries with two bound shared attributes, each by all $n$ factors. Therefore, the right hand side of \eqref{Proter+} reduces to $n+2(n-2)=3n-4$. When $R$ has a $2$-key, there are only $n-2$ factors, and the shared attributes are now free, so the count changes to $n-2+2(n-2-1)=3n-8$.
\end{proof}
\noindent It is interesting that there is no drop in ternarity bound when the relation already has a $1$-key as opposed to just admitting one. There is a change from projoin to join, but all the ternaries come from the branch point that has valency $n$ in both cases. Other special constructions also provide upper bounds. For example, it follows from \eqref{NotInProRedEq} that $\ter(\neg I_n)\leq 3n-6$ for $n\geq4$ and $|\D|\geq n-1$. Bounds from below are conceptually harder because we have to rule out all bonds with fewer ternaries as reductions. We will obtain such a bound for non-degenerate relations in the next section by exploiting graph-theoretic properties of bond diagrams.

Upon reflection, absence of reducible factors in a reduction is too weak a property. Not only can complete reductions have redundant predicates, as long as those are irreducible, but they may not be minimal. For example, four teridentities bonded in a square reduce $I_4$ to irreducibles, but hypostatic abstraction gives a minimal reduction with only two teridentities. 

On the other hand, minimal reductions can be incomplete for only trivial reasons. And they can be converted into complete reductions by a trimming procedure that is reflected in diagrams by {\it merging} of predicate vertices. In algebraic terms, when two factors share bound attribute(s) we replace them in the bond by their relative product. Merging cannot be used on a pair of ternaries with a single shared attribute, because it creates a quaternary, but in all other cases the bond remains subternaric. For example, the bond on Figure \ref{BondGraphs}\,b) can be merged into a single binary. The next lemma uses merging to produce complete minimal reductions, and gives additional support to discounting unaries and binaries in ternarity counts. 
\begin{lemma}\label{MonFac} Let $R$ be a subternarily reducible relation.

\noindent \textup{(i)} In any minimal reduction of $R$ any two factors share at most one attribute, i.e. the bonding diagram has no multiedges. 

\noindent \textup{(ii)} If $R$ does not have an unary Cartesian factor then its minimal reductions have no unaries at all.

\noindent \textup{(iii)} If $R$ does not have a binary Cartesian factor then there exist its minimal reductions with no binaries at all.
\end{lemma}
\vspace{-1em}
\begin{figure}[!ht]
\begin{centering}
a)\ \includegraphics[width=0.25\textwidth]{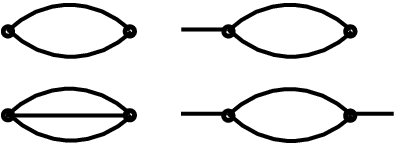} \hspace{0.1in} b)\ \includegraphics[width=0.21\textwidth]{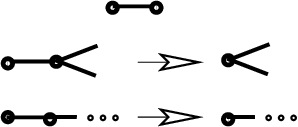} \hspace{0.1in} c)\ \includegraphics[width=0.30\textwidth]{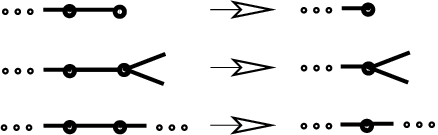} \par
\end{centering}
\caption{\label{Merging} Schematic bond diagrams, the dots stand for predicate vertices: a) subcubic multiedges; b) merging unaries; c) merging binaries.}
\end{figure}
\vspace{-1em}
\begin{proof} \textup{(i)} Possible multiedge configurations are shown on Figure \ref{Merging}\,a). Two of them have no loose ends and cannot occur because minimal reductions are non-redundant. In the other two merging would eliminate a ternary, so they cannot occur by minimality.

\textup{(ii)} If the unary shares its attribute with another factor there are three cases. It is another unary and they form a redundant component, which is ruled out. It is a ternary and merging will turn it into a binary, which is also ruled out. Finally, if it is a binary then merging will reduce it to an unary, and we can repeat the process, Figure \ref{Merging}\,b). By induction on the number of binaries, it must stop, and it can only stop when the unary's attribute is free (we hit a loose end). But then $R$ has an unary Cartesian factor, contrary to the assumption. 

\textup{(iii)} By (i), a binary can share at most one attribute with another factor, and when it does, merging reduces the number of binaries, Figure \ref{Merging}\,c). By induction, all binaries can be eliminated except for those with both attributes free, i.e. binary Cartesian factors.
\end{proof}
\noindent After applying the Lemma's procedure, the only reducible factors left, if any, are degenerate binary Cartesian factors. Factoring them into pairs of unary factors produces a complete minimal reduction.

\section{Ternarity and Peirce's reduction thesis}\label{ternPRT}

In this section we will use ternarity and graph theory to refine Peirce's reduction thesis on infinite domains, and show that its strengthened form fails dramatically on finite domains.
\begin{definition}\label{bondgraph}
The \textbf{bond graph} is obtained from the bond diagram by placing additional vertices at the loose ends and removing the labels.
\end{definition}
\noindent The bond graph is just a multigraph of graph theory \cite{Berge}. If the bond was subternaric then the multigraph will be {\it subcubic} (subtrivalent), i.e. have vertices of degree at most $3$. Such graphs are widely studied, particularly due to applications in structural chemistry and knot theory. By Lemma \ref{MonFac}, when the bond is a minimal reduction of a relation without unary Cartesian factors, the bond graph is a simple graph and its vertices of degree $1$, called {\it pendants} in graph theory, are in $1$-$1$ correspondence with the relation's attributes. 

To prove the next lemma, we will need a graph-theoretic formula originally due to Listing \cite{List}. Let $V$ be the number of vertices, $E$ the number of edges, $C$ the number of fundamental cycles, and $K$ the number of connected components, then $V-E+C-K=0$. This formula is often used in modern graph theory as the definition of $C$, called the cyclomatic number, which is then proved to be equal to the number of fundamental cycles \cite{Berge}. 
\begin{lemma}\label{III-I} If $I,I\!I,I\!I\!I$ denote the numbers of vertices of degrees $1, 2, 3$, respectively, in a subcubic multigraph then $I\!I\!I-I = 2(C-K)$. In particular, $I$ and $I\!I\!I$ have the same parity. If the graph is connected with at least $n$ pendants then $I\!I\!I\geq n-2$. 
\end{lemma}
\begin{proof} In a subcubic multigraph we have $I+I\!I+I\!I\!I=V$. And, by the handshaking theorem, the sum of all vertex valencies is twice the number of its edges, so $I+2I\!I+3I\!I\!I=2E$. Therefore, $V-E=\frac12(I-I\!I\!I)$. It remains to multiply both sides by $2$ and note that $V-E=-(C-K)$ by the Listing's formula. In a connected multigraph $K=1$, so $I\!I\!I=I-2+2C\geq n-2$.
\end{proof}
\noindent It is instructive to give a direct intuitive argument for the case $n=3$. Since the multigraph is connected there exist paths going from two of the pendants to the third. They must meet at some vertex, and then proceed jointly to the destination (of course, they may meet and diverge several times). That meeting vertex must have degree $3$.
\begin{corollary}\label{AdTerPar} 
If $R$ does not have an unary Cartesian factor then its ternarity and arity have the same parity.
\end{corollary}
\begin{proof} We have $\ter (R)=I\!I\!I$ in the bond graph of a minimal bond reduction of $R$. By Lemma \ref{MonFac}\,(ii), it contains no unaries, so the only pendants come from loose ends, and $I$ is the arity of $R$. The result now follows directly from Lemma \ref{III-I}. 
\end{proof}
\noindent The no-unary-factor condition cannot be dropped, the quaternary $P(u)\land Q(x,y,z)$ with non-degenerate $Q$ has arity $4$ but ternarity $3$.

Peirce's reduction thesis is essentially equivalent to $1\leq\ter(R)<\infty$ for non-degenerate $R$ with $n\geq3$. The next theorem gives us the exact number, on infinite domains, and the promised irreducibility of non-degenerate ternaries, on any domains.
\begin{theorem}\label{TerNonDeg} 
Ternarity of any non-degenerate $n$-ary $R$ with $n\geq2$ satisfies $\ter(R)\geq n-2$. In particular, any non-degenerate ternary is bond irreducible. If, moreover, $|R|\leq|\D|$ then $\ter(R)=n-2$. In particular, $\ter(R)=n-2$ for all relations on infinite domains.
\end{theorem}
\begin{proof} Consider a bond decomposition of $R$. Since $R$ is non-degenerate its bond multigraph is connected, and has at least $n$ pendants coming from the loose ends, the free variables. Therefore, $I\geq n$. The lower bound $\ter(R)\geq n-2$ now follows directly from Lemma \ref{III-I}. For $n=3$ this means that any bond decomposition of a non-degenerate ternary must contain a ternary, i.e. such ternaries are bond irreducible.

If $|R|\leq|\D|$ then, by Lemma \ref{KeyAdmCrit}, it admits a $1$-key, and, by Theorem \ref{TerKeyRed}, $\ter(R)\leq n-2$, hence $\ter(R)=n-2$. On infinite domains $|R|\leq|\D|$ holds for all relations.
\end{proof}
\noindent The next example shows that the inequality in the lower bound on $\ter$ can be strict on finite domains.
\begin{example}[\textbf{Herzberger's quaternary}] Consider the quaternary $H$ on $\D=\{\a,\b,\g\}$ introduced by Herzberger in \cite{Herz} and given by the table below.
\medskip

$H$ := \begin{tabular}{|C|C|C|C|}
\hline
 \a & \b & \b & \a \\  \hline
 \b & \a & \a & \b \\ \hline
 \g & \b & \g & \b \\ \hline
 \b & \g & \b & \g \\  \hline
\end{tabular}\ \ \ \ \ 
$H^{1,2,3}$ = \begin{tabular}{|C|C|C|}
\hline
 \a & \b & \b\\  \hline
 \b & \a & \a \\ \hline
 \g & \b & \g \\ \hline
 \b & \g & \b \\ \hline
 \end{tabular}\ \ 
$H^{1,2,4}$ = \begin{tabular}{|C|C|C|}
\hline
 \a & \b & \a\\  \hline
 \b & \a & \b \\ \hline
 \g & \b & \b \\ \hline
 \b & \g & \g \\ \hline
 \end{tabular}
 \medskip
 
\noindent One can check by cases that $H$ is non-degenerate. We have $|H|=4>3=|\D|$, and Herzberger shows by combinatorial search that $H$ is, indeed, not a relative product of two ternaries (actually, he only shows that for one partition of  attributes, but the argument works analogously for others). Therefore, $\ter(H)>2$. 

However, $H$ is not a counterexample to reducibility. One can see by inspection that it has a $2$-key (in fact, any two of its columns are a $2$-key). Therefore, by Theorem \ref{KeyRedRel}, it is a join of two ternaries, e.g. $H=H^{1,2,3}\Join H^{1,2,4}$ if we pick the first two columns as the $2$-key. Bond explication \eqref{multiidn} of the two shared attributes converts this join into a bond of {\it four} ternaries, its projections $H^{1,2,3}$, $H^{1,2,4}$, and two teridentities. Since $\ter(H)$ must be even by Lemma \ref{III-I} and $\ter(H)>2$ we conclude that $\ter(H)=4$.
\end{example}
Herzberger's observation can be strengthened by relating ternarity to the number of parameters in projoin reductions and applying Theorem \ref{ProJoinRedLim}. It turns out that linear bounds on ternarity in terms of arity, as in Theorem \ref{TerKeyRed}, are not typical for general relations. However, we cannot infer existence of relations of infinite ternarity, i.e. of irreducible $n$-aries with $n\geq4$, and hence refute Peirce's original thesis.
\begin{theorem}\label{TerRedLim} The share of $n$-ary relations with $n\geq4$ and $\ter(R)\leq m$ among all such relations on a domain $\D$ is $<1$ for $|\D|>\binom{\frac{3m+n}{2}}{n-1}$, and asymptotically vanishes when $|\D|\to\infty$. In particular, there exist $n$-ary relations of arbitrarily high ternarity. 
\end{theorem}
\begin{proof} Suppose $R$ is non-degenerate with $\ter(R)\leq m$. By Lemma \ref{MonFac}, there is a purely ternary minimal reduction of it. In its bond graph $I=n$, $I\!I\!I=\ter(R)$, and $2E=I+3I\!I\!I$ by the handshaking theorem. Of the edges, $n$ are incident to pendants and correspond to free variables, while $k:=E-n$ correspond to bond parameters. Therefore, our minimal reduction is a projoin reduction with $k=\frac{3\,\ter(R)-n}{2}\leq\frac{3m-n}{2}$ parameters. 

By Theorem \ref{ProJoinRedLim}, the share of relations projoin reducible with $k$ parameters is $<1$ when $|\D|>\binom{n+k}{n-1}$ and it goes to $0$ when $|\D|\to\infty$. Since $n+k\leq\frac{3m+n}{2}$ this inequality is satisfied for our $k$. Degenerate relations are projoin reducible with even $0$ parameters, let alone $k$, so there must be non-degenerate relations with $\ter>m$ on our domain. Moreover, their share approaches $1$ when $|\D|\to\infty$.
\end{proof}
\noindent The estimate we used in the theorem is very rough. Indeed, $k$ is not the number of parameters in just any projoin reduction, but in a complete reduction down to ternaries. One could merge ternaries, as long as the merged factors still have arity $<n$, and reduce that number. To get a more accurate estimate one can count the number of subcubic graphs with $n$ pendants, $m$ cubic and no degree $2$ vertices, and bound the number of relations that have them as their bond graphs.

Finally, Peirce's reduction thesis on infinite domains and bond explication show that of all ternaries only one is needed in reductions -- teridentity. It is to unaries, binaries and teridentities that we should aim to reduce all relations. This suggests our next definition.
\begin{definition}\label{I3ter+}
$\boldsymbol{I_3}$\textbf{-ternarity} of a relation, denoted $\ter_{\!I_3}$, is the minimal number of teridentities in its subternaric bond reductions where the only ternaries are teridentities, and $\infty$ if no such reductions exist. 
\end{definition}
\noindent Clearly, $\ter_{\!I_3}\leq\ter$ and $\ter_{\!I_3}=\ter$ on infinite domains because any ternary decomposes into binaries and teridentities by hypostatic abstraction. The next example shows that on finite domains, again, the inequality can be strict.
\begin{example} Suppose a non-degenerate ternary has a subternaric decomposition with a single teridentity. Since it has no unary factors unaries can be eliminated, and we are left with the teridentity with up to three chains of binaries attached to it in the graph. A chain of binaries can be merged into a single one by relative products, and represent our ternary as a teridentity directly bonded with three (or fewer) binaries. Turning the teridentity into a branch point we obtain a projoin of binaries with a single parameter. 

However, we showed in Example \ref{NotI3ProIrred} that $\neg I_3$ on a domain with $|\D|=2,3$ is not projoin reducible with one parameter. Therefore, while $\ter(\neg I_3)=1$ trivially, $\ter_{\!I_3}(\neg I_3)>1$. Since $\ter_{\!I_3}(\neg I_3)$ must be odd we can conclude that $\ter_{\!I_3}(\neg I_3)\geq3$, i.e. it takes at least $3$ teridentities to bond $\neg I_3$ on small domains, if it is possible at all.
\end{example}

\section{Conclusions and open problems}

We studied reduction of relations to relations of smaller arity under three relational operations: join, projoin and bond. All three can be expressed by conjunctions and existential quantification on predicates, and are motivated by algebraic analogies and practical applications in the database theory. Aside from unifying and extending known reduction results and constructions, we described the sets of irreducible relations and the structure of complete reductions to them. 
We also clarified the relationship between projoin and bond reducibility, and the import of Peirce's reduction thesis. Finally, we introduced the notion of ternarity that, intuitively, measures complexity of `relating' in a relation, and used it to sharpen reducibility results.

Aside from concrete results, a major takeaway from this work is the striking gap between reduction behavior on finite and infinite domains. As far as we know, the only author to notice the phenomenon before was Herzberger \cite{Herz}. We showed that the gap gets wider as the size of the domain grows: the share of irreducible relations with bounded number of parameters (Theorem \ref{ProJoinRedLim}) or with bounded ternarity (Theorem \ref{TerRedLim}), grows with it, even though it is $0$ at $\infty$. The root cause of this discrepancy is the equality $|\D|=|\D|^2$ for infinite cardinalities, which is equivalent to the axiom of choice. 

This raises a big question: to what extent does Peirce's reduction thesis hold on finite domains? While non-degenerate ternaries are still irreducible there (even by stronger means than bonds and projoins \cite{CorPos04}), reduction is obstructed by the lack of enough domain elements for classical constructions. 
\begin{quote} {\bf Problem 1:} Are there irreducible $n$-ary relations with $n\geq4$? 
\end{quote} 
Such relations would have to have a lot of tuples, $|R|>|\D|^{n-2}$. Otherwise, they will have a $k$-key with $k\leq n-2$ and hypostatic abstraction will reduce them (Theorem \ref{ProKeyRed}). Counting arguments we used would not settle the question alone, because general projoins and bonds do not have a finite combinatorial description like
Cartesian products, joins, projoins with bounded number of parameters, or bonds with bounded ternarity. On the other hand, general tests of irreducibility, like the ones for join irreducibility in Theorem \ref{JoinIrredTests}, also seem to be elusive. A promising approach is provided by the clone theory, where one can dualize the problem into one about functional clones via the Pol-Inv Galois connection. Projoin bases of small arity for maximal  sub-co-clones of the co-clone of all relations on $2$-element domains are constructed in \cite{BRSV05}. If one could construct bases containing only unaries, binaries and ternaries for the co-clone of all relations on any finite domain that would resolve the question negatively.

While we suspect a negative answer to the first problem, it is more likely to be affirmative for the next one. Although non-degenerate ternaries are trivially `reducible' to themselves on any domains, there is a non-trivial reducibility question about them. To resolve it negatively, one would need a basis of the co-clone of all relations containing only unaries and binaries. 
\begin{quote} {\bf Problem 2:} Are there ternary relations indecomposable into bonds of unaries, binaries and teridentities (equivalently, projoin irreducible to unaries and binaries)? 
\end{quote} 
We already saw that the strongest form of the reduction thesis, that gives ternarity of non-degenerate $n$-ary relations as $n-2$, fails on finite domains. Even if reductions are always possible their complexity must be higher than that of their infinite counterparts. In particular, there can be no bound on ternarity in terms of arity alone. However, since there are finitely many $n$-aric relations of finite ternarity on a finite domain ternarity must attain a maximum on them. 
\begin{quote} {\bf Problem 3:} Find sharp upper bounds on ternarity in terms of arity and the size of the domain. 
\end{quote} 
We did not address the question of uniqueness, but a relation can have purely ternaric minimal reductions whose bond graphs are not even isomorphic. Indeed, bonding teridentities on any cubic graph with $n$ hanging edges produces $n$-identity, and it is easy to construct non-isomorphic tree graphs with equal numbers of cubic vertices. Perhaps, this diversity is due to overabundance of symmetry in $I_n$.  
\begin{quote} {\bf Problem 4:} Are the bond graphs of purely ternaric minimal reductions unique for `generic' non-degenerate relations?
\end{quote} 
One can think of minimal reductions as revealing the structure of information processing within a relation, which suggests an affirmative answer. Attributes of a relation generalize inputs and outputs of a function, and functions are commonly interpreted as information processors \cite[1.5.4]{Lau}. In \cite{Mal} a measure of information exchange among relation's attributes is introduced and studied, similar measures are studied in computational biology \cite{Teg16}. The intuition of ternarity as `complexity of relating attributes' suggests a connection.
\begin{quote} {\bf Problem 5:} Is there an information-theoretic interpretation of ternarity, e.g. bounds on measures of information exchange in terms of it? 
\end{quote}
To summarize, logical factorization of relations poses many interesting challenges at the intersection of mathematical logic, combinatorics, graph theory and data science.

\end{document}